\documentclass[11pt,reqno]{amsart}
\usepackage{amsbsy,amsfonts,amsmath,amssymb,amscd,amsthm}
\usepackage{mathrsfs,float,color}
\usepackage[mathcal]{euscript}
\usepackage{subfigure,enumitem,graphicx,epstopdf}
\usepackage[a4paper,text={6.1in,8in},centering,top=1.5in]{geometry}
\usepackage{pxfonts,epstopdf}
\usepackage[colorlinks=true,linkcolor=blue,citecolor=green]{hyperref}
\usepackage{srcltx}
\usepackage[margin=20pt,font=small,labelfont=bf,textfont=it]{caption}
\usepackage[latin1]{inputenc}

\hypersetup{linktocpage}

\small\normalsize
\theoremstyle{plain}
\newtheorem{mainthm}{Theorem}

\newtheorem{thm}{Theorem}[section]

\newtheorem{lem}[thm]{Lemma}
\newtheorem{claim}[thm]{Claim}
\newtheorem{prop}[thm]{Proposition}

\newtheorem{defi}[thm]{Definition}

\theoremstyle{definition}

\newtheorem{rem}[thm]{Remark}

\newcommand{\eqdef}{\stackrel{\scriptscriptstyle\rm def}{=}}

\makeatletter
\def\l@part{\@tocline{0}{-2pt}{1pc}{}{}}
\def\l@section{\@tocline{1}{-2pt}{1pc}{4.6em}{}}
\renewcommand{\tocpart}[3]{%
  \indentlabel{\@ifnotempty{#2}{\makebox[2.3em][l]{%
    \ignorespaces#1 #2.\hfill}}}\bf{#3}}
\renewcommand{\tocsection}[3]{%
  \indentlabel{\@ifnotempty{#2}{\hspace*{2.3em}\makebox[2.3em][l]{%
    \ignorespaces#1 #2.\hfill}}}#3}

\makeatother

\newcommand{\Z}{\mathbb{Z}}
\newcommand{\R}{\mathbb{R}}

\newcommand{\N}{\mathbb{N}}

\let\oldtocsection=\tocsection
\let\oldtocsubsection=\tocsubsection
\renewcommand{\tocsection}[2]{\hspace{0em}\bf\oldtocsection{#1}{#2}}
\renewcommand{\tocsubsection}[2]{\hspace{4.8em}\oldtocsubsection{#1}{#2}}
\let\oldtocsubsubsection=\tocsubsubsection
\renewcommand{\tocsubsubsection}[2]{\hspace{4.2em}\oldtocsubsubsection{#1}{#2}}
\DeclareMathOperator{\spn}{span}

\newenvironment{psmallmatrix}
  {\left(\begin{smallmatrix}}
  {\end{smallmatrix}\right)}

\begin{document}
~\vspace{-0.5cm}
\title[Extremal exponents]{Extremal exponents of random products of conservative diffeomorphisms}
\author[Barrientos]{Pablo G. Barrientos}
\address{\centerline{Instituto de Matem\'atica e Estat\'istica, UFF}
    \centerline{Rua M\'ario Santos Braga s/n - Campus Valonguinhos, Niter\'oi,  Brazil}}
\email{pgbarrientos@id.uff.br}
\author[Malicet]{Dominique Malicet}
\address{LAMA, Université Paris-Est Marne-la-Vallée, Batiment Coppernic
5 boulevard Descartes} \email{mdominique@crans.org}

%\subjclass[2010]{Primary 58F15, 58F17; Secondary: 53C35.}
%\keywords{}
%\thanks{MTM2014-56953-P}

\begin{abstract}
We show that for a $C^1$-open and $C^{r}$-dense subset of the set
of ergodic iterated function systems of conservative
diffeomorphisms of a finite-volume manifold of dimension~$d\geq
2$, the extremal Lyapunov exponents do not vanish. In particular,
the set of non-uniform hyperbolic systems contains a $C^1$-open
and $C^r$-dense subset of ergodic random products of independent
conservative surface diffeomorphisms.
\end{abstract}
\setcounter{tocdepth}{2} \maketitle
%\tableofcontents
~\vspace{-1.0cm} \thispagestyle{empty}

\section{Introduction}
The notion of uniform hyperbolicity introduced by Smale
in~\cite{Sm67} was early shown to be less generic than initially
thought~\cite{AS70,New70}. In order to describe a large set of
dynamical systems, Pesin theory~\cite{Pe77} provides a weaker
notion called \emph{non-uniform hyperbolicity}. These systems are
described in terms of non-zero Lyapunov exponents of the
\emph{linear cocycle} defined by the derivative transformation,
the so-called \emph{differential cocycle}. In contrast with the
non-density of hyperbolicity, we had to wait some decades to
construct the first examples of systems with robustly zero
Lyapunov exponents~\cite{KN07,BBD16}. Even in the conservative
setting there are open sets of smooth diffeomorphisms with
invariant sets of positive measure where all the Lyapunov
exponents vanish identically~(see~\cite{CS89,Her90,Xia92}).
However, recently in~\cite{LY17} it was showed that conservative
diffeomorphisms without zero exponent in a set of positive volume
are $C^1$-dense.

On the other hand, abundance of non-uniform hyperbolicity has been
obtained in the general framework of linear cocycles when the base
driving dynamics is fixed and the matrix group is perturbed in
many different contexts~\cite{F63,K92,AC97,V08,A11}. However,
nothing is known for random product of independent non-linear
dynamics. That is, for cocycles driving by a shift map endowed
with a Bernoulli probability to value in the group of
diffeomorphisms of a compact manifold. To perturbe the Lyapunov
exponents of these cocycles one must change the non-linear
dynamics similar to the case of differential cocycles. Examples of
iterated function systems (IFSs) of diffeomorphisms with robust
zero extremal Lyapunov exponents with respect to some ergodic
measure that not project on a Bernoulli measure were provided
in~\cite{gorodetski2005nonremovable,BBD14}. The authors
in~\cite{BBD14} question if such examples of IFS with robust zero
extremal Lyapunov exponents could be constructed by taking the
dynamics conservative and the ergodic measure defined as the
product measure of a shift invariant measure on the base and the
volume measure on the fiber. We give a partial negative answer to
this question by showing that the non-uniform hyperbolic systems
contain a $C^1$-open and $C^r$-dense subset of ergodic IFSs
generated by independent conservative surface
$C^r$-diffeomorphisms. In higher dimension we get the same result
for the extremal Lyapunov exponents.

\subsection{Random products of conservative diffeomorphisms}
An \emph{iterated function system} (IFS) can also be thought of as
a finite collection of functions which can be applied successively
in any order. We will focus on the study of IFSs generated by
$C^r$-diffeomorphisms $f_1,\dots,f_k$ of a finite-volume
Riemannian manifold $M$ which preserve the normalized Lebesgue
measure $m$. We will denote
$$
f^0_\omega=\mathrm{id}, \quad f_\omega^n=f_{\omega_{n-1}}\circ
\dots \circ f_{\omega_0}  \quad \text{for
    $\omega=(\omega_i)_{i\geq 0}\in\Omega_+\eqdef \{1,\dots,k\}^\mathbb{N}$ \ \text{and} \ $n>0$.}
$$

According to the random Oseledets multiplicative
theorem~\cite{LM06} there are real numbers $\lambda_-(x)\leq
\lambda_+(x)$ called \emph{extremal Lyapunov exponents} such that
for $\mathbb{P}_+$-almost every $\omega\in \Omega_+$ and
$m$-almost every $x\in M$,
$$
\lambda_-(x)=\lim_{n\to +\infty}\frac{1}{n}\log
\|Df_\omega^n(x)^{-1}\|^{-1} \quad \text{and} \quad
\lambda_+(x)=\lim_{n\to +\infty}\frac{1}{n}\log
\|Df_\omega^n(x)\|.
$$
Here $\mathbb{P}_+=p^\mathbb{N}$ is a Bernoulli measure on
$\Omega_+$ where $p=p_1\delta_1 +\dots +p_k \delta_k$ with $p_i>0$
and $p_1+\dots+p_k=1$. We will assume that $m$ is an ergodic
measure for the group generated by $f_1,\dots,f_k$. This means
that any strictly $f_i$-invariant measurable subset of $M$ for all
$i=1,\dots,k$ has either null or co-null measure. Hence,
$\lambda_\pm=\lambda_\pm(x)$ are constant $m$-almost everywhere.
Thus, $\lambda_\pm$ only depends on the conservative
$C^r$-diffeomorphisms $f_1,\dots, f_k$. Actually, these Lyapunov
exponents also depend on the Bernoulli probability but we will
consider the weights $p_i$ fixed and thus we will not explicit
this dependance. Thereby we will denote
$\lambda_\pm=\lambda_\pm(f_1,\dots,f_k)$. Moreover, since the
diffeomorphisms are conservative, we have that the sum of all
Lyapunov exponents must be zero, so that in particular,
$\lambda_-\leq 0\leq \lambda_+$.

Roughly speaking, we expect that for a generic choice of the
$f_1,\dots,f_k$ should exhibit some hyperbolicity, measured by the
non-nullity of the extremal Lyapunov exponents. In dimension one,
preserving Lebesgue measure forces that the group generated by
$f_1,\dots,f_k$ is conjugate to a subgroup rigid circle rotations
and thus, $\lambda_+=\lambda_-=0$. In fact, it is well-know that
the random independent products of finitely many diffeomorphisms
of the circle have non-null (actually, negative) Lyapunov
exponent, unless if all the diffeomorphisms preserve a common
probability measure (c.f.~\cite{malicet2017random}). Thus, from
now on, we will assume that the finite-volume Riemannian manifold
$M$ has dimension $d\geq 2$. Our main result is the following:

\begin{mainthm}\label{main}
    Given $r\geq 1$ and $k\geq 2$, consider conservative
    $C^r$-diffeomorphisms $f_1,\ldots,f_{k-1}$ of $M$
    such that the Lebesgue measure is ergodic for the
    group generated by these maps. Then there is a $C^1$-open
    and $C^r$-dense set $\mathcal{U}$ of conservative
    $C^r$-diffeomorphisms such that for any $f_k$  in $\mathcal{U}$,
    $$\lambda_-(f_1,\ldots,f_k)<0<\lambda_+(f_1,\ldots,f_k).$$
\end{mainthm}

After the conclusion of this work, Obata and Poletti sent us a
preprint~\cite{OP18} where they get a similar result using
different approach. They proved that the set of systems with
positive integrated extremal Lyapunov exponent contains a
$C^1$-open and $C^1$-dense subset of random products of
i.i.d.~conservative diffeomorphisms of a compact connected
oriented~surface.
%We expect that for $r,k\geq 2$, there is an open
%and dense set
%of $k$-tuples $(f_1,\dots,f_k)$ of conservative
%$C^r$-diffeomorphisms of compact surface so that the Lebesgue
%measure is ergodic for the group action generated by these maps.
%of ergodic IFSs generated by $k$ conservative
%$C^r$-diffeomorphisms of compact surface.
%In particular, this conjecture imply that the non-uniform
%hyperbolic systems contains an open and dense set of iterated
%function systems  generated by $(k+1)$-tuples of
%$C^r$-conservative diffeomorphisms.

The scheme of the proof of Theorem~\ref{main} is the following.
Let us denote $\mathbb{P}(TM)$ the projective tangent space of
$M$. Given a diffeomorphism $f$ the differential cocycle $(f,Df)$
naturally acts on $\mathbb{P}(TM)$, and by a slight modification
we can actually see $(f,Df)$ as a cocycle acting on $M\times
\mathbb{P}(\mathbb{R}^d)$. By using this observation, we will
deduce that if
$\lambda_-(f_1,\ldots,f_k)=\lambda_+(f_1,\dots,f_k)=0$, then the
cocycles $(f_i,Df_i)$ share a common invariant measure on $M\times
\mathbb{P}(\mathbb{R}^d)$ projecting on $m$. Then, up to some
measurable modification of the $k-1$ first cocycles, one can
characterize (using the ergodicity assumption) the invariant
measures on $M\times \mathbb{P}(\mathbb{R}^d)$ by $(f_i,Df_i)$
projecting on $m$ for all $i=1,\dots,k-1$ as product measures.
This property is sufficiently rigid so that we will manage to
break it by a perturbation of the $k$-th diffeomorphism,
concluding thus that $\lambda_-(f_1,\ldots,f_k)$ needs to be
different from $\lambda_+(f_1,\dots,f_k)$. The fact that the
transformations preserve the Lebesgue measure only appears at the
time of performing this last perturbation which prevents
$(f_i,Df_i)$ from having a common invariant measure for all
$i=1,\dots,k$. The main step is the second one, that is, the
classification of the invariant measure of measurable cocycles
explained below in a more general setup.

\subsection{Measurable random cocycles}
Consider an invertible measure preserving transformation $f$ of a
standard Borel probability space $(X,\mu)$. Let $\mathrm{G}$ be a
locally compact topological group whose operation is denoted by
juxtaposition. Given a measurable function $A:X\to \mathrm{G}$ we
define the  \emph{$\mathrm{G}$-valued cocycle} over $f$ by the
dynamical defined products
$$
A^0(x)=\mathrm{id}, \quad  A^n(x) = A(f^{n-1}(x))\cdots
A(f(x))A(x) \quad \text{and} \quad A^{-n}(x)=A^n(f^{-n}(x))^{-1}
\quad \text{if \ $n>0$}.
$$
As usual, we denote this cocycle by $(f,A)$. We say that $(f,A)$
is a \emph{linear cocycle} if $\mathrm{G}$ is a subgroup of the
group $\mathrm{GL}(d)$ of invertible $d\times d$ matrices. %with real coefficients.
We can always neglect sets of null measure, and we shall identify
cocycles which coincide $\mu$-almost surely.

There is a natural group structure on the set of
$\mathrm{G}$-cocycles over invertible transformations on
$(X,\mu)$. The product of two cocycles $(f,A)$ and $(g,B)$ is
defined by
$$
(g,B)\cdot(f,A)=(g\circ f,(B\circ f) A).
$$
In particular, the powers of $(f,A)$ are given by
$(f,A)^n=(f^n,A^n)$. Thus, the natural way to study the action of
several cocycles is the notion of \emph{cocycle over a group}. If
$\mathrm{T}$ is a group of measure preserving transformations of
$(X,\mu)$, the classical definition (see~\cite{Z84}) of a
$\mathrm{G}$-valued cocycle over $\mathrm{T}$ is a Borel function
$\alpha:\mathrm{T}\times X\rightarrow \mathrm{G}$ such that
$$
\alpha(ts,x)=\alpha(t,s(x))\alpha(s,x) \quad \text{for all \
$t,s\in \mathrm{T}$  \ and \  $\mu$-almost every \ $x\in X$.}
$$
Then, given invertible $\mu$-preserving transformations $f_i: X\to
X$ and measurable functions $A_i:X\to \mathrm{G}$ for all $i\in
Y$, one can define a cocycle $\alpha$ over  the  group
$\mathrm{T}$ generated by these transformations such that
$\alpha(f_i,x)=A_i(x)$.  However, we will not use this formalism.
With our terminology, by denoting $A_t(x)=\alpha(t,x)$ for $t\in
\mathrm{T}$, we identify $\alpha$ with the group of cocycles
$\hat{\mathrm{T}}=\{(t,A_t): t\in \mathrm{T}\}$. In order to study
the action of several cocycles, instead of looking a cocycle over
a group, we will actually prefer to use the the formalism of
\textit{random dynamical systems}, by defining the notion of
\emph{random cocycles} as \emph{random walks} on the group
$\hat{\mathrm{T}}$. To do this, we will start by considering a
particular class of cocycles.

\subsubsection{Random cocycle} Denote by $\Omega$
the product space $Y^\mathbb{Z}$ endowed with the product
mea\-sure $\mathbb{P}=p^\mathbb{Z}$ where $(Y,p)$ is some
probability space. Let $\theta:\Omega\to \Omega$ be the shift map
on $\Omega$. Set $\bar{X}=\Omega\times X$ and
$\bar{\mu}=\mathbb{P}\times \mu$. A \emph{random transformation}
is a measurable invertible skew-shift
\begin{equation*} \label{random-map}
  \bar{f}:  \bar{X} \to \bar{X},  \quad
  \bar{f}(\bar{x})=(\theta\omega,f^{}_\omega(x)) \quad
  \text{and  \ $\bar{x}=(\omega,x)\in \bar{X}$}
\end{equation*}
where $f_\omega=f_i$ depends only on the zeroth coordinate
$\omega_0=i$ of $\omega=(\omega_n)_{n\in\mathbb{Z}}\in\Omega$. We
will assume that $\bar{\mu}$ is a $\bar{f}$-invariant ergodic
measure. Observe that the invariance of this measure implies that
$\mu$ is a $f^{}_i$-invariant measure $p$-almost surely
(cf.~\cite{LM06}). Finally,
\begin{defi}
We say that $(\bar{f},A)$ is a \emph{random $\mathrm{G}$-valued
cocycle} if $A:\bar{X} \to \mathrm{G}$ defines a cocycle over a
$\bar{\mu}$-preserving ergodic random transformation
$\bar{f}:\bar{X}\to \bar{X}$ such that
$$
A(\bar{x})=A_i(x) \quad \text{for $\bar{\mu}$-almost every \
$\bar{x}=(\omega,x)\in \bar{X}=\Omega\times X$ \ with \
$\omega=(\omega_n)_{n\in\mathbb{Z}}$ \ and \ $\omega_0=i$.}
$$
\end{defi}

Obviously, a deterministic cocycle is a particular case of random
cocycle by taking $\Omega$ being a one-point space.  Thus, we can
also see the set of random cocycles as an extension of the
deterministic case.

\subsubsection{Cohomologous random cocycles} Two
$\mathrm{G}$-valued cocycles $(f,A)$ and $(f,B)$ over the same
$\mu$-preserving invertible transformation $f:X\to X$ are called
\emph{cohomologous} if there exists a measurable map $P:X\to
\mathrm{G}$ such that
$$
   B(x)=P(f(x))^{-1}A(x)P(x) \quad \text{for $\mu$-almost all
   $x\in X$.}
$$
In other words, $(f,B)=\phi_P^{-1} \circ (f,A) \circ \phi^{}_P$
where $\phi^{}_P$ denotes the cocycle $(\mathrm{id},P)$. Thus the
relation of cohomology defines a equivalence between measurable
cocycles. We are interesting to study the measurable cohomological
reduction of random cocycles. First, observe that the cohomology
class of a random cocycle could contain non-random cocycles. To
keep it into the class of random cocycles we need to ask that the
conjugacy $P:\bar{X}\to \mathrm{G}$ actually does not depend on
$\omega\in \Omega$. That is,
\begin{defi} Two random $\mathrm{G}$-valued
cocycles $(\bar{f},A)$ and $(\bar{f},B)$ are cohomologous if and
only if there is a measurable function $P:X\to \mathrm{G}$ such
that
$$
B(\bar{x})=P(f_\omega(x))^{-1}A(\bar{x})P(x) \quad \text{for
$\bar{\mu}$-almost every $\bar{x}=(\omega,x)\in
\bar{X}=\Omega\times X$.}
$$
Equivalently, if for $p$-almost every $i\in Y$ it holds that
$$
B_i(x)=P(f_i(x))^{-1}A_i(x)P(x) \quad \text{ $\mu$-almost every $x\in
X$.}
$$
\end{defi}

\subsubsection{Invariant measure of random cocycles}
Let $(\bar{f},A)$ be a random $\mathrm{G}$-valued cocycle. Since
these cocycles are locally constant, i.e., only depends on the
zeroth coordinate of the random sequence, for each $i\in Y$ we
have a $\mathrm{G}$-valued cocycle $(f_i,A_i)$. Denote by $Z$ a
$\mathrm{G}$-space. That is, a topological space where
$\mathrm{G}$ acts by automorphisms. The cocycles $(f_i,A_i)$
naturally act on $X\times Z$ by means of the skew-product maps
$$
  F_i\equiv(f_i,A_i):X\times Z \to X\times Z , \quad F_i(x,z)=
(f_i(x),A_i(x)z).
$$
\begin{defi} A measure $\hat\mu$ on $X\times Z$ is called
\emph{$(\bar{f},A)$-invariant} if $\hat\mu$ projects on $\mu$ and
it is $F_i$-invariant for $p$-almost every $i\in Y$. Similarly,
$\hat\mu$ is called \emph{$(\bar{f},A)$-stationary} if $\hat\mu$
project on $\mu$ and
$$
     \hat\mu \ast p \eqdef \int_Y F_i\hat\mu \, dp(i)=\hat\mu.
$$
\end{defi}

\subsection{Characterization of invariant measures of random linear cocycles}
We are interested in the projective invariant measures of random
linear cocycles $(\bar{f},A)$. That is, we want to understand the
$(\bar{f},A)$-invariant measure on $X \times Z$ of a random
$\mathrm{G}$-valued cocycles where $\mathrm{G}=\mathrm{GL(d)}$ and
$Z=\mathbb{P}(\mathbb{R}^d)$ is the projective space of
$\mathbb{R}^d$. This problem was previously studied by Arnold,
Cong and Oseledets in~\cite{ACO97} using ingredients of the Zimmer
amenable reduction theorem~\cite{Z84}. They proved that up to
cohomology all the projective invariant measures are deterministic
measures:

\begin{thm}[{\cite[Thm.~3.8]{ACO97}}]\label{main-cocycle}
Let $(\bar{f},A)$ be a  random $\mathrm{GL}(d)$-valued cocycle.
Then there is a random linear cocycle $(\bar{f},B)$ cohomologous
to $(\bar{f},A)$ such that every $(\bar{f},B)$-invariant measure
$\hat{\mu}$ on $X\times \mathbb{P}(\mathbb{R}^d)$ is a product
measure, i.e., it is of the form $\hat{\mu}=  \mu\times \nu$ where
$\nu$ is a measure on $\mathbb{P}(\mathbb{R}^d)$.
\end{thm}

The above result has gone something unnoticed in the literature on
cocycles.  In fact, we only found this result once we had obtained
a similar characterization of the invariant measures for random
$\mathrm{GL}(2)$-valued cocycles. Ours approach to classify the
invariant measures is substantially different from the proof
provides in~\cite{ACO97}. For instance, we do not use Zimmer
amenable reduction theory. Motivated by this difference, in order
to provide an alternative complete proof we have tried to extend
our ideas for random $\mathrm{GL}(d)$-valued cocycle with $d>2$.
Our approach consists in dividing the random cocycles into two
disjoint classes defined below:

\begin{defi}
    A random $\mathrm{G}$-valued cocycle $(\bar{f},A)$ is called
    \emph{essentially bounded} if there are a compact set
    $K$ of $\mathrm{G}$, a constant $\delta>0$ and a subset
    $E\subset \bar{X}$ of positive $\bar{\mu}$-measure
 satisfying that
 $$
 d(\{n\geq 0: A^n(\bar{x})\in  K\})\geq \delta \quad
 \text{for all $\bar{x}\in E$.}
 $$
where $d$ denotes the lower asymptotic density on
$\mathbb{N}\cup\{0\}$.
%a set in
%$$
% d(D_K)\eqdef \liminf_{n\to \infty} \frac{\#\big([0,n]\cap D_K\big)}{n} =0.
%$$
Otherwise, we say that $(\bar{f},A)$ is \emph{essentially
unbounded}.
\end{defi}

For essentially bounded random $\mathrm{GL}(d)$-valued cocycles we
have got generalized our ideas to
obtain~Theorem~\ref{main-cocycle}. In fact, this result will be
consequence of the following more general theorem for random
cocycles with values in a locally compact topological group
$\mathrm{G}$:

\begin{mainthm} \label{thm:measure-product-group}
Let $(\bar{f},A)$ be an essentially bounded random
$\mathrm{G}$-valued cocycle. Then there is a  random cocycle
$(\bar{f},B)$ with values on a compact subgroup $\mathrm{H}$ of
$\mathrm{G}$ such that
\begin{enumerate}
\item $(\bar{f},A)$ is cohomologous with $(\bar{f},B)$, and
\item  $\mu \times
m_{\mathrm{H}}$ is the unique $(\bar{f},B)$-stationary probability
measure on $X\times \mathrm{H}$.
\end{enumerate}
Moreover, any $(\bar{f},B)$-stationary probability measure
$\hat\pi$ on $X\times \mathrm{G}$ is a product measure of the form
$\hat\pi=\mu \times \nu$ with $\nu=m_H \ast \omega$ where $\omega$
is a measure on $\mathrm{G}$ and $m_\mathrm{H}$ is the Haar
measure on $\mathrm{H}$. In particular, every stationary
probability measure is, in fact, invariant.
\end{mainthm}

The compactness of $\mathrm{H}$ implies that must be essentially
contained in the maximal compact subgroup of $\mathrm{G}$. That
is, we can find $Q\in \mathrm{G}$ such that $Q^{-1}\mathrm{H}Q$ is
contained in the maximal compact subgroup of $\mathrm{G}$. Thus,
we get the following remark:

\begin{rem} \label{rem:O(2)}
Up to conjugacy, if $\mathrm{G}=\mathrm{GL}(d)$ we have that
$\mathrm{H} \subset \mathrm{O(d)}$ where $\mathrm{O}(d)$ denotes
the ortogonal group of $d\times d$ matrizes with real
coeficientes.
\end{rem}

Additionally, we characterize the essentially bounded cocycle as
indicated in the following result:

\begin{mainthm} \label{thm:equiv}
    Let $(\bar{f},A)$ be a random $\mathrm{G}$-valued cocycle.
    Then, it is  equivalent: \vspace{-0.1cm}
    \begin{enumerate}
        \item \label{item-eb} $(\bar{f},A)$ is essentially bounded,
        \item \label{item-st} there is a $(\bar{f},A)$-stationary probability measure on $X\times
        \mathrm{G}$,
        \item \label{item-cp} $(\bar{f},A)$ is cohomologous to a random cocycle with values on a compact subgroup  of $\mathrm{G}$.
    \end{enumerate}
\end{mainthm}

Unfortunately we have not been able to generalize our ideas to
prove Theorem~\ref{main-cocycle} for essentially unbounded
$\mathrm{GL}(d)$-cocycle with $d>2$. In the last section of this
paper we include our proof of Theorem~\ref{main-cocycle} for
essentially unbounded random $\mathrm{GL}(2)$-valued cocycles. In
this way, all the results in this paper provide a self-contained
proof of the main result (Theorem~\ref{main}) in dimension two.

\subsection{Invariance principle for random linear cocycles}
Let $(\bar{f},A)$ be a random linear cocycle.
When $\log^+ \|A^{\pm 1}\|$ are both $\bar{\mu}$-integrable
functions, by Furstenberg-Kesten theorem (see~\cite{Via14}) there
are real numbers $\lambda_-(A) \leq \lambda_+(A)$, called
\emph{extremal Lyapunov exponent of $(\bar{f},A)$}, such that
$$
\lambda_-(A)=\lim_{n\to +\infty}\frac{1}{n}\log
\|A^n(\bar{x})^{-1}\|^{-1} \quad \text{and} \quad
\lambda_+(A)=\lim_{n\to +\infty}\frac{1}{n}\log \|A^n(\bar{x})\|
$$
for $\bar{\mu}$-almost every $\bar{x}\in\bar{X}$. An
\emph{invariance principle} is an statement on the rigidity of the
projective invariant measure under the assumption of the
coincidence of extremal Lyapunov exponents. As a consequence of an
invariance principle of Ledrappier~\cite{Le86} (see
also~\cite{C90,Via14}) and the classification of invariant measure
obtained in Theorem~\ref{main-cocycle}, we will get the following:

\begin{mainthm} \label{thm3}
Let $(\bar{f},A)$ be a random $\mathrm{GL}(d)$-valued cocycle
satisfying the integral conditions. If
$\lambda_-(A)=\lambda_+(A)$, then there is a random linear cocycle
$(\bar{f},B)$ cohomologous to $(\bar{f},A)$ and a probability
measure $\nu$  on $\mathbb{P}(\mathbb{R}^d)$ such that
$B_i(x)\nu=\nu$ for $p$-almost every $i\in Y$ and $\mu$-almost
every $x\in X$.
\end{mainthm}

\begin{rem} \label{rem-final}
In Theorem~\ref{main-cocycle} and~\ref{thm3}  the  random
$\mathrm{GL}(d)$-valued cocycles $(\bar{f},A)$ and  $(\bar{f},B)$
are cohomologous by means of a measurable function $P: X\to
\mathrm{SL}^\pm(d)$. Here $\mathrm{SL}^\pm(d)$ denotes the group
of $d\times d$ matrices with determinant $\pm 1$.
%This follows
%from the fact that, actually, we get that the random
%$\mathrm{SL}^\pm(d)$-valued cocycle $(\bar{f},A_0)$ where $A_0=
%|\mathrm{det} A|^{-1/2} A$ is cohomologous to a random
%$\mathrm{SL}^\pm(d)$-valued cocycle $(\bar{f},B_0)$  being
%$B=|\mathrm{det}\,  A|^{1/2} \cdot B_0$.
\end{rem}

\subsection{Organization of the paper} In the  following
section we will prove Theorem~\ref{main} using
Theorem~\ref{main-cocycle} and assuming Theorem~\ref{thm3}. The
proof of Theorem~\ref{thm3} will be provided in
Section~\S\ref{sec:invariance}. In Section~\S\ref{sec:bounded} we
will prove Theorem~\ref{thm:measure-product-group},
Theorem~\ref{thm:equiv} and will provide an alternative proof of
Theorem~\ref{main-cocycle} for essentially bounded random
$\mathrm{GL}(d)$-valued cocycles. Finally, in
Section~\S\ref{sec:unbounded} we study essentially unbounded
$\mathrm{GL}(2)$-valued linear cocycles including an alternative
proof of Theorem~\ref{main-cocycle} in this case.

\section{Extremal exponents: Proof  of Theorem~\ref{main}}
\label{sec:main proof} We fix a finite-volume Riemannian manifold
$M$ of dimension $d\geq 2$.  At several places, we will need to
look the matrix of a differential of a diffeomorphism of $M$ in
some basis. In view of this, we denote by $\mathscr{B}$ the set of
measurable maps $P$ on $M$ such that $P(x)$ is an invertible linear map from
$T_xM$ to $\R^d$ for any $x\in M$. We also denote by $\mathscr{O}$
the subset of $\mathscr{B}$ such that for any $P\in\mathscr{O}$ we
have that $P(x)$ is an isometry (where $\R^d$ is endowed with the
Euclidean norm and $T_xM$ with the Riemannian structure). Then,
given $P\in\mathscr{B}$ and a $C^r$-diffeomorphism $f$ preserving
the normalized Lebesgue measure $m$, we define
$$
J_Pf(x)= P(f(x))\circ Df(x)\circ P(x)^{-1} \quad \text{for all
$x\in M$.}
$$
Thus, $J_Pf(x)$ belongs to $\mathrm{GL}(d)$ and actually to
$\mathrm{SL}^\pm(d)$ if $P\in \mathscr{O}$. Similarly we define
the class
$$
  \mathscr{S}=\{P\in \mathscr{B}: \, P =  R Q \ \text{with} \
  R\in \mathrm{SL}^\pm(d) \ \text{and} \ Q \in \mathscr{O}\}.
$$
Notice that $\mathscr{O}\subset \mathscr{S}\subset \mathscr{B}$
and $J_Pf(x)\in \mathrm{SL}^\pm(d)$ for all $P\in
\mathscr{S}$.

\subsection{Invariance principle} We fix $k\ge 2$ and let
$f_1,\ldots,f_k$ be $C^r$-diffeomorphisms of $M$ preserving  $m$.
Assume that the group generated by these maps is ergodic with
respect to  $m$.  Consider a probability measure
$p=p_1\delta_1+\dots+p_k\delta_k$ on $Y=\{1,\dots,k\}$ and set
$\mathbb{P}=p^\mathbb{Z}$ on $\Omega=Y^\mathbb{Z}$. Let $\bar{f}$
be the skew-shift on $\bar{M}=\Omega\times M$ given by
$\bar{x}=(\omega,x)\mapsto (\theta\omega,f_\omega(x))$ where
$f_\omega=f_i$ if the zeroth coordinate of
$\omega=(\omega_n)_{n\in\mathbb{Z}}$ is $\omega_0=i$. Observe that
$\bar{f}$ preserves the measure $\bar{m}=\mathbb{P}\times m$. For
a fixed $P\in\mathcal{B}$ we define the measurable map
$$
J_P\bar{f}:\bar{M}\to \mathrm{GL}(d), \quad
J_P\bar{f}(\bar{x})=J_Pf_\omega(x) \ \ \text{for
$\bar{x}=(\omega,x)\in \bar{M}$.}
$$
Then, $(\bar{f},J_P\bar{f})$ is a random $\mathrm{GL}(d)$-valued
cocycle.

As we did in the introduction, we define the extremal Lyapunov
exponents $\lambda_\pm(f_1,\ldots,f_k)$. Notice that if
$\lambda_-(f_1,\dots,f_k)=\lambda_+(f_1,\dots,f_k)$, then
$\lambda_-(J_P\bar{f})=\lambda_+(J_P\bar{f})$ for all
$P\in\mathcal{O}$. Thus, as a consequence of Theorem~\ref{thm3}
(actually of Proposition~\ref{step1}), we have the following
result.
\begin{prop}\label{step2}
Assume that $\lambda_-(f_1,\dots,f_k)=\lambda_+(f_1,\dots,f_k)$.
Then
%there is $\bar{P} \in \mathcal{S}$ such that
%$(\bar{f},J_{\bar{P}}\bar{f})$ has an invariant product
%probability measure on $M\times \mathbb{P}(\mathbb{R}^d)$, i.e.,
%of the form $m\times \nu$. Consequently,
for every $P\in
\mathscr{B}$, the random linear cocycle  $(\bar{f},J_P\bar{f})$
admits an invariant measure on $M\times \mathbb{P}(\R^d)$.
\end{prop}
\begin{proof}
Let $Q\in \mathscr{O}$. Consider the random
$\mathrm{SL}^\pm(d)$-valued cocycle $(\bar{f},J_Q\bar{f})$. From
the assumption we obtain that
$\lambda_-(J_Q\bar{f})=\lambda_+(J_Q\bar{f})$. Hence, according to
Theorem~\ref{thm3},  we find a random cocycle $(\bar{f},B)$
cohomologous to $(\bar{f},J_Q\bar{f})$ with a product invariant
measure on $M\times \mathbb{P}(\R^d)$. In particular
$(\bar{f},J_Q\bar{f})$ has an invariant measure. This proves the
proposition since for any $P\in \mathscr{B}$,
$J_Pf_i(x)=R^{-1}(f_i(x))J_{Q}f_i(x)R(x)$ where
$R(x)=Q(x)P(x)^{-1}$.  %in the above argument.
%Moreover, according to
%Remark~\ref{rem-final}, there is $R:M\to \mathrm{SL}^\pm(d)$ such
%that $ B_i(x)=R(f_i(x))^{-1}J_Q\bar{f}(x)R(x)$ for $m$-almost
%every $x\in M$ and every $i=1,\dots,k$. In particular, taking
%$\bar{P}(x)=R(x)^{-1}Q(x)\in \mathscr{S}$ we have that
%$B_i(x)=J_{\bar{P}}f_i(x)$. Thus, $B=J_{\bar{P}}\bar{f}$ and we
%conclude the proof of the proposition.
\end{proof}

%\begin{rem}
%The consequence in the above proposition is immediate since for
%any $P\in \mathscr{B}$,
%$J_Pf_i(x)=R^{-1}(f_i(x))J_{\bar{P}}f_i(x)R(x)$ where
%$R(x)=\bar{P}(x)P(x)^{-1}$. In fact, this consequence can be
%obtained directly by applying Proposition~\ref{step1} instance of
%Theorem~\ref{thm3} in the above argument.
%\end{rem}

\subsection{Breaking the invariance by perturbation}

Now, we are going to prove the following:% proposition:
\begin{prop}\label{step3}
    For any $P\in \mathscr{S}$, the set $\mathcal{R}_P$ of the maps
    $f$ in $\mathrm{Diff}_m^r(M)$ such that $(f,J_Pf)$
    has no invariant measure on $M\times \mathbb{P}(\R^d)$ of
    the form $m\times\nu$ is open and dense in $\mathrm{Diff}_m^r(M)$.
\end{prop}
The openness of $\mathcal{R}_P$ is obvious so we focus in proving
the density. A small issue is the a priori absence of continuity
points of $P$. We bypass the problem by defining points of
regularity of $P$ in a weaker sense: by Lusin Theorem, there
exists a increasing sequence $(E_j)_{j\in\N}$ of measurable
subsets of $M$ whose Lebesgue measure goes to $1$ such that $P$
restricted to $E_j$ is continuous. Moreover, up to replace $E_j$
by the subset of its density points, we can assume that every
point of $E_j$ is a density point. Finally we set $E=\cup_j E_j$.
The set $E$ has full measure and will be considered as the
regularity points of $P$. The following lemma express the way we
will use this regularity:

\begin{lem} \label{pseudo-cont} Let $f$ be a conservative
$C^r$-diffeomorphism of $M$ and consider $x_0 \in E\cap
f^{-1}(E)$. If the linear cocycle $(f,J_Pf)$ has an invariant
measure of the form $m\times \nu$, then $J_Pf(x_0)\nu=\nu$.
\end{lem}

\begin{proof}
The assumption implies that $J_Pf(x)\nu=\nu$ for $m$-almost every
$x\in M$.   In particular this equality holds for $m$-almost every
$x$ in  $E_j$. Take $j$  large enough so that $x_0$ is a density
point of $E_j\cap f^{-1}(E_j)$. Since $J_Pf$ is continuous on this
set, the equality holds for $x=x_0$.
\end{proof}

Now, for $\alpha=(a,a',b,b')$ in $E^4$ with $a\not=b$ we consider
the set
$$\mathcal{D}_\alpha=\{ f \in \mbox{Diff}_m^r(M): \, f(a)=a' \mbox{ and } f(b)=b' \}.$$
We defined also the application
$$
\Phi: \mathcal{D}_\alpha \to \mathrm{SL}^\pm(d)\times
\mathrm{SL}^\pm(d), \qquad  \Phi(f)=(J_Pf(a),J_Pf(b)).
$$
The application $\Phi$ is indeed valued in $\mathrm{SL}^\pm(d)\times \mathrm{SL}^\pm(d)$ because $P\in
\mathscr{S}$. We have the following properties.

\begin{lem} \label{lem-claim}
Consider $f\in \mathcal{D}_\alpha$ and let $\mathcal{V}$ be a
neighborhood of $f$ in $\mathrm{Diff}^r_m(M)$. Then
$\Phi(\mathcal{V}\cap \mathcal{D_\alpha})$ has non-empty interior
and  $\Phi(\mathcal{V}\cap \mathcal{D}_\alpha \setminus
\mathcal{R}_P)$ has empty interior.
\end{lem}

\begin{proof}
We write $g=f\circ u$ for the elements of $\mathcal{V}\cap
\mathcal{D}_\alpha$. The map $u$ runs over the set
$\mathcal{V}_0\cap \mathcal{D}_0$ where $\mathcal{V}_0$ is some
neighborhood of the identity in $\mathrm{Diff}^r_m(M)$ and
$$
\mathcal{D}_0=\{u\in \mbox{Diff}_m^r(M): \, u(a)=a \ \text{and} \
u(b)=b\}.
$$
A
simple computation gives
$$\Phi(g)=(J_Pf(a)P(a)Du(a)P(a)^{-1},J_Pf(b)P(b)Du(b)P(b)^{-1}).$$
Now, since $\mathrm{SL}^\pm(d)$ is a normal subgroup of
$\mathrm{GL}(d)$ and $a\not = b$,  the range of the map
$$u \in \mathcal{V}_0\cap
\mathcal{D}_0 \mapsto (P(a)Du(a)P(a)^{-1},P(b)Du(b)P(b)^{-1})$$
contains an open ball in $\mathrm{SL}^\pm(d)\times
\mathrm{SL}^\pm(d)$. Hence, the range of $\Phi$ restricted to
$\mathcal{V}\cap \mathcal{D}_\alpha$ also contains an open set.
This concludes the first item.

We will prove now that $\Phi(\mathcal{V}\cap
\mathcal{D}_\alpha\setminus \mathcal{R}_P)$ has empty interior in
$\mathrm{SL}^\pm(d)\times \mathrm{SL}^\pm(d)$. Indeed, if $f$
belongs to the complement of $\mathcal{R}_P$, then $(f,J_Pf)$ has
an invariant measure of the form $m\times \nu$. Moreover, if $f\in
\mathcal{D}_\alpha$, then, by Lemma~\ref{pseudo-cont}, we have
that $J_Pf(a)\nu=J_Pf(b)\nu=\nu$. In particular
$\Phi(\mathcal{V}\cap\mathcal{D}_\alpha\setminus \mathcal{R}_P)$
is included in the set of pair of matrices $(A,B)$ in $
\mathrm{SL}^\pm(d)\times \mathrm{SL}^\pm(d)$ such that there is a
probability measure $\nu$ on $\mathbb{P}(\R^d)$ satisfying that
$A\nu=B\nu=\nu$. From~\cite[Prop.~8.14 and Rem.~8.15]{ASV13} (see
also~\cite[Sec.~7.4.2]{Via14}) this set has empty interior in
$\mathrm{SL}^\pm(d)\times \mathrm{SL}^\pm(d)$  and consequently
$\Phi(\mathcal{V}\cap\mathcal{D}_\alpha\setminus \mathcal{R}_P)$
also has it. This completes the proof.
%(TODO:proof)
\end{proof}

 Using the above lemma, we can prove the density of $\mathcal{R}_P$:

 \begin{proof}[Proof of Proposition~\ref{step3}]
 Let $\mathcal{V}$ be any open ball of
 $\mathrm{Diff}_m^r(M)$. Consider $f\in \mathcal{V}$
 and take any pair of different points $a$, $b$ in $E\cap f^{-1}(E)$.
 Setting $\alpha=(a,f(a),b,f(b))$ we have that
 $\mathcal{V}\cap \mathcal{D}_\alpha$ is not empty (it contains $f$).
 From Lemma~\ref{lem-claim}, we obtain that $\mathcal{V}\cap \mathcal{D}_\alpha$
 cannot be included in $\mathcal{D}_\alpha\setminus \mathcal{R}_P$.
 Hence $\mathcal{V}$ must intersect $\mathcal{R}_P$.
 Thus $\mathcal{R}_P$ is dense. As we previously mentioned,
 $\mathcal{R}_P$ is also open and then we  conclude the proof
 of the proposition.
\end{proof}

\subsection{Proof of Theorem~\ref{main}}
Let $f_1,\ldots,f_{k-1}$ be $C^r$-diffeomorphisms of $M$
preserving $m$ and such that its group action is ergodic.
Consider $Q\in\mathscr{O}$ and take the random cocycle
$(\bar{f}^*,J_Q\bar{f}^*)$ defined by these $k-1$ maps. According
to Theorem~\ref{main-cocycle}, we find a random  cocycle
$(\bar{f}^*,B)$ cohomologous to $(\bar{f}^*,J_Q\bar{f}^*)$ such
that every $(\bar{f}^*,B)$-invariant measure on $M\times
\mathbb{P}(\R^d)$ is a product measure. Moreover, according to
Remark~\ref{rem-final}, there is $R:M\to \mathrm{SL}^\pm(d)$ such
that $ B_i(x)=R(f_i(x))^{-1}J_Qf_i(x)\,R(x)$ for $m$-almost every
$x\in M$ and every $i=1,\dots,k-1$. In particular, taking
$P(x)=R(x)^{-1}Q(x)\in \mathscr{S}$ we have that
$B_i(x)=J_{P}f_i(x)$. Thus, $B=J_{P}\bar{f}^*$.

By Proposition~\ref{step3}, one can find a dense open set
$\mathcal{R}_P$ such that $(f,J_Pf)$ has no invariant measure on
$M\times \mathrm{P}(\R^d)$ of the form $m\times\nu$. In
consequence, for every $f_k$ in $\mathcal{R}_P$ the maps
$(f_i,J_Pf_i)$ for $i=1,\dots,k$ have not a common invariant
probability measure projecting on $m$  of the form $m\times\nu$.
Moreover, the random cocycle $(\bar{f},J_P\bar{f})$ defined by
theses $k$ maps does not admit any invariant probability measure.
Indeed, if $\hat\mu$ is a $(\bar{f},J_P\bar{f})$-invariant
probability measure also it is
$(\bar{f}^*,J_P\bar{f}^*)$-invariant. Consequently $\hat\mu=m
\times \nu$ and thus $(f_k,J_Pf_k)$ has a product invariant
measure which is not possible. Hence, according to
Proposition~\ref{step2}, $\lambda_+(f_1,\ldots,f_k)\not =
\lambda_+(f_1,\dots,f_k)$. This concludes Theorem~\ref{main}.

\section{Invariance principle: Proof of Theorem~\ref{thm3}}
\label{sec:invariance} Let $(\bar{f},A)$ be a random
$\mathrm{GL}(d)$-valued cocycle. Suppose that $\log^+ \|A^\pm\|$
are both $\bar{\mu}$-integrable functions
 and recall the definition of
the extremal Lyapunov exponents $\lambda_\pm(A)$ given in the
introduction.

\begin{prop}\label{step1}
    Let $(\bar{f},A)$ be a random cocycle as above and assume that
    $\lambda_-(A)=\lambda_+(A)$. Then there exists a
    $(\bar{f},A)$-invariant
    probability measure $\hat{\mu}$ on $X\times \mathbb{P}(\R^d)$.
%    That is, $\hat\mu$ projects on $X$ over $\mu$ and
%    $F_i\hat{\mu}=\hat\mu$ for $p$-almost
%    every $i\in Y$ where $F_i\equiv(f_i,A_i): (x,E) \mapsto
%    (f_i(x),A_i(x)E)$ on $X\times\mathbb{P}(\R^d)$.
\end{prop}
\begin{proof}
Since $\mathbb{P}(\mathbb{R}^d)$ is a compact metric space, by
Proposition~\ref{prop-anex1} in Appendix~\ref{Apendix1} we can
take a $(\bar{f},A)$-stationary probability measure $\hat\mu$ on
$X\times \mathbb{P}(\mathbb{R}^d)$. We are going to prove that
under the assumption $\lambda_-(A)=\lambda_+(A)$, the probability
$\hat{\mu}$ is in fact invariant. This measure have some
disintegration $d\hat{\mu}=\nu_x \, d\mu(x)$. We want to prove
that this measure is invariant which is equivalent to say that
$\nu_{f_i(x)}=A_i(x)\nu_x$ for $\mu$-almost every $x$ and
$p$-almost every $i\in Y$.

Let us define the non-invertible skew-shift associated with the
random dynamics. We set $\Omega_+=Y^\N$, $\mathbb{P}_+=p^\N$ and
$\theta:\Omega_+\rightarrow\Omega_+$ the shift map. Then we set
$\bar{X}_+=\Omega_+\times X$, $\bar{\mu}_+=\mathbb{P}_+\times \mu$,
and $\bar{f}_+:\bar{X}_+\rightarrow \bar{X}_+$ given by
$\bar{x}=(\omega,x)\mapsto (\theta\omega,f_{i}(x))$ where
$\omega=(\omega_n)_{n\geq 0}$ and $\omega_0=i$. The triplet
$(\bar{X}_+,\bar{\mu}_+,\bar{f}_+)$ is a non-invertible ergodic
measure preserving dynamical system. Observe that since
$A:\bar{X}\to \mathrm{GL}(d)$ is locally constant then  it also
define a $\mathrm{GL}(d)$-valued cocycle over this non-invertible
system which we simply denote by $(\bar{f}_+,A)$.  Notice that, by
the assumption, the extremal Lyapunov exponents of this new linear
cocycle are also coincident.

Now we use the non-invertible version of the invariance principle
of Ledrappier~\cite[Prop.~2 and Thm.~3]{Le86} (see
also~\cite[Thm.~7.2]{Via14}). First, we consider the action of
this cocycle $(\bar{f}_+,A)$ on $\bar{X}_+\times
\mathbb{P}(\R^d)$. Then, since $\hat\mu$ is an
$(\bar{f},A)$-stationary measure then the measure
$\eta_+=\mathbb{P}_+\times \hat{\mu}_+$ is invariant for this action
(see~\cite{LM06}). Its projection on $\bar{X}_+$ is $\bar{\mu}_+$,
and the corresponding disintegration is $d\eta_+=
\nu_{\bar{x}}\,d\bar{\mu}(\bar{x})$ where $\nu_{\bar{x}}=\nu_x$
only depends on the second coordinate of $\bar{x}=(\omega,x)\in
\Omega_+ \times X=\bar{X}_+$. The invariance principle of
Ledrappier says that coincidence of the extremal Lyapunov
exponents implies the equality
$\nu_{\bar{f}_+(\bar{x})}=A(\bar{x})\nu_{\bar{x}}$ for
$\bar{\mu}_+$-almost every $\bar{x}=(\omega,x)\in \bar{X}_+$. This
is equivalent to say that $\nu_{f_i(x)}=A_i(x)\nu_x$ for
$\mu$-almost every $x\in X$ and $p$-almost every $i\in Y$. This
concludes the proof.
\end{proof}

As a corollary of the above proposition and
Theorem~\ref{main-cocycle} we get Theorem~\ref{thm3}:

\begin{proof}[Proof of Theorem~\ref{thm3}] If $(\bar{f},A)$ is a
random $\mathrm{GL}(d)$-valued cocycle satisfying the
integrability conditions and  $\lambda_-(A)=\lambda_+(A)$,
according to Proposition~\ref{step1} we get a
$(\bar{f},A)$-invariant probability measure $\hat\mu$ on
$X\times\mathbb{P}(\mathbb{R}^d)$. From
Theorem~\ref{main-cocycle}, there is a random linear cocycle
$(\bar{f},B)$ cohomologous to $(\bar{f}, A)$ such that every
$(\bar{f},B)$-invariant probability measure is a product measure.
In particular, the $(\bar{f},A)$-invariant measure $\hat\mu$ is
transported to a $(\bar{f},B)$-invariant measure $\hat\mu_1$. Then
we get that $\hat\mu_1=\mu\times \nu$ where $\nu$ is a probability
measure on $\mathbb{P}(\R^d)$. Consequently $B_i(x)\nu=\nu$ for
$p$-almost every $i\in Y$ and $\mu$-almost $x\in X$. This
completes the proof.
\end{proof}

%\section{Classification of invariant measures: Proof of
%Theorem~\ref{main-cocycle}} \label{sec:cocycle}

\section{Essentially bounded cocycles} \label{sec:bounded}

In this section we are going to study the invariant measures of
essentially bounded random cocycles. First, we prove
Theorem~\ref{thm:measure-product-group} after that we get
Theorem~\ref{thm:equiv} and finally we classify the invariant
measures of essentially bounded random linear cocycles.

\subsection{Proof of Theorem~\ref{thm:measure-product-group}}
%Now we will prove Theorem~\ref{thm:measure-product-group}.
We will split the proof in three steps.

\subsubsection{Step 1:~Existence of stationary measure} First we
will study  the existence of stationary measures for essentially
bounded random cocycle.

\begin{prop} \label{medida}
    Let $(\bar{f},A)$ be an essentially bounded random
    $\mathrm{G}$-valued cocycle.
    Then there exists a $(\bar{f},A)$-stationary Borel probability
    regular measure $\hat{\mu}$ on $X\times \mathrm{G}$.
\end{prop}

\begin{proof}
Set $\delta_e$ be the probability measure on $\mathrm{G}$
supported on the identity element $e$.  Let us consider the
sequence $(\hat\nu^n)_n$ of probability measures $\hat\nu^{n}$ on
$X\times \mathrm{G}$ defined by means of the disintegration
$d\hat{\nu}^n = \nu^n_x \, d\mu(x)$ where $\nu^n: x \mapsto
\nu^n_x$ is given by
$$
  \nu^{n} = \frac{1}{n} \sum_{j=0}^{n-1}
  \mathscr{P}^{*j}\delta_{e} \in
  L^\infty(X;\mathscr{P}(\mathrm{G}))
$$
being $\mathscr{P}^*$ the adjoint transfer operator on
$L^\infty(X;\mathscr{M}(\mathrm{G}))$ introduced in
Appendix~\ref{Apendix1}. According to Lemma~\ref{lem:Apendix1},
there exists  a finite regular Borel $(\bar{f},A)$-stationary
measure $\hat\nu$ on $X\times G$ which is an accumulation point in
the weak$^*$ topology of $(\hat\nu^n)_n$. Notice
    that the mass of the probability measures $\hat{\nu}_{n}$
    could be escaping to
    infinite and thus $\hat{\nu}$ could be zero. However,
    this is not the case, since $(\bar{f},A)$ is essentially bounded.
    Indeed, there are a compact set $K$ of $\mathrm{G}$ and
    a constant $\delta>0$  satisfying that
    $D_K(\bar{x})=\{m\geq 0: A^m(\bar{x})\in  K\}$ has $\delta$ lower
 density on a set $E$ of positive $\bar{\mu}$-measure.
Then for any $n\geq 0$,
    \begin{align*}
    %\label{non-zero-measure}
    \hat{\nu}^n(X\times K) &= \int  \nu^n_x(K) \, d\mu(x) =
 \frac{1}{n} \sum_{j=0}^{n-1} \int \int A^{-j}(\omega,x)^{-1}\delta_e(K) \,
 d\mathbb{P}(\omega)d\mu(x)
    \\
&=\frac{1}{n} \sum_{j=0}^{n-1} \int  A^{j}(\bar{f}^{-j}(\bar
x))\delta_e(K) \,
 d\bar\mu(\bar x) = \frac{1}{n} \sum_{j=0}^{n-1} \int  A^{j}(\bar{y})\delta_e(K) \,
 d\bar\mu(\bar y)
     \\
    &= \int \frac{\#\,\{0\leq j\leq n-1: A^{j}(\bar{y}) \in K \}}{n} \,
    d\bar{\mu}(\bar{y}) \\
    &= \int \frac{\#\, ([0,n-1]\cap
        D_K(\bar{y}))}{n} \, d\bar{\mu}(\bar{y}),
    \end{align*}
so that by Fatou lemma
$$\hat\nu(X\times K)=\liminf_{n\to +\infty} \hat{\nu}^n(X\times K)\geq \int \liminf_{n\to +\infty} \frac{\#\, ([0,n-1]\cap
        D_K(\bar{y}))}{n} \, d\bar{\mu}(\bar{y})\geq \delta\bar{\mu}(E),$$
    hence $\hat\nu$ is not
    equal to zero. Therefore, we get that the normalized measure
    is a probability measure.

    Finally, we will prove that the normalized measure of $\hat\nu$
    projects on $\mu$. To see this, first we observe that $\pi \hat\nu
    (B)  \leq \mu(B)$ for all measurable set
    $B\subset X$ where $\pi$ is the projection onto $X$. Indeed,  %notice that $\pi^{-1}(B)=B\times \mathrm{G}$.
    let $E\subset B \subset V$ be, respectively, a compact set and an
    open set of $X$ approximating the $\mu$-measure of $B$. Hence
    \begin{align*}
    \hat\nu(E\times \mathrm{G}) &\leq \hat\nu(V\times \mathrm{G}) \leq
    \liminf_{n\to\infty} \hat{\nu}^n(V\times \mathrm{G}) =
    %\liminf_{n\to \infty} \, \frac{\#\,\{0\leq n\leq N-1: f^n_\omega(x)\in V  \}}{N} =
    \mu(V).
    \end{align*}
    The second inequality holds from the
    weak$^*$ convergence~(c.f.~\cite[Thm.~1 Sec.~1.9]{EG}). Thus,
    since  both, $\hat\nu$ and   $\mu$, are regular measures we get
    that
    \begin{align*}
    \pi\hat\nu(B)=\hat\nu(B\times\mathrm{G})&=\sup
    \{ \, \hat\nu(E\times K) : \ \text{$E$ and $K$ compact and
        $E \times K  \subset B\times \mathrm{G}$}
    \} \\
    &\leq \inf \{ \, \mu(V) : \ \text{$V$ open  and $B  \subset V$} \}=\mu(B).
    \end{align*}
    Therefore $\pi\hat\nu$ is absolute continuous with respect to
    $\mu$.
    Also $\pi\hat\nu$ is $(\theta,f)$-stationary since $\pi\circ F_i = f_i\circ \pi$
    for all $i\in Y$ where recall that $F_i\equiv(f_i,A_i):X\times \mathrm{G}\to X\times \mathrm{G}$.
    Moreover, $\pi\hat\nu$ is proportional to $\mu$. Indeed,
    according
    to Radon-Nikodym theorem
    (see~\cite[thm.~13.18]{guide2006infinite})
    we can
    writing $d\pi
    \hat\nu=\phi(x)\,d\mu(x)$ with $\phi\in L^1(\mu)$.
    The ergodicity
    of $\bar{\mu}$ implies that $\phi$ is constant $\mu$-almost everywhere: indeed, by Birhoff theorem, for any bounded measurable $u:X\rightarrow \R$,
    $$\int u(x) \phi(x) d\mu(x)=\int\left(\frac{1}{n}\sum_{j=0}^{n-1}u(\bar{f}^n(\omega,x))\right)\phi(x)d\bar{\mu}(\omega,x)\xrightarrow[n\to+\infty]{}\left(\int u d\mu\right)\left(\int \phi d\mu\right),$$
    so that we obtain in fact that $\phi(x)=\pi\hat\nu(X)=\int \phi d\mu$ for $\mu$-almost
    every $x\in X$. Therefore, we get that the normalized
    measure of $\hat\nu$ is a $(\bar{f},A)$-stationary Borel probability
    regular measure %which projects on $\mu$
    and the proof of the proposition is completed.
\end{proof}

\subsubsection{Step 2:~Reduction of the cocycle} Consider $h\in
\mathrm{G}$ and define
$$
\Phi_h: X\times \mathrm{G} \to X\times \mathrm{G}
\quad \Phi_h(x,g)=(x,gh).
$$
\begin{prop} \label{lem-reduction}
    Let $\hat\pi$ be a $(\bar{f},A)$-stationary ergodic Borel probability regular measure
    on $X\times \mathrm{G}$. %projecting on $\mu$.
    Then,
    $\mathrm{H}(\hat{\pi})=\{h\in \mathrm{G}:
    \Phi_h\hat\pi=\hat\pi\}$   is a compact subgroup of $\mathrm{G}$ and there exists a
    measurable function $P: X \to \mathrm{G}$  such that
    $$
    \text{$B(\omega,x)=P(f_\omega(x))^{-1}A(\omega,x)P(x) \in
        \mathrm{H}(\hat\pi)$ \quad for $(\mathbb{P}\times\mu)$-almost
        every $(\omega,x)\in\Omega\times X$.}
    $$
\end{prop}

\begin{proof}
First of all, notice that $\mathrm{H}\equiv\mathrm{H}(\hat\pi)$ is
group since
    $(\Phi_h)^{-1}=\Phi_{h^{-1}}$ and $\Phi_{h\ell}=\Phi_{\ell}\Phi_{h}$.
    Also, $\mathrm{H}$ is closed because the map $h \mapsto
    \Phi_h\hat\pi$ is continuous. Then, to get that $\mathrm{H}$ is compact
    we only need to prove that $\mathrm{H}$ is bounded.

    By contradiction, if $\mathrm{H}$ is not bounded, then
    for any pair of compact sets $K$ and $L$  of $\mathrm{G}$
    there is $h\in \mathrm{H}$ such that $L \cap Kh =\emptyset$.
    Otherwise, $H\subset K^{-1}L$ which is a compact set.
    Since $\hat\pi$ is a
    Borel regular probability measure on $X\times\mathrm{G}$
    we can find a compact set $K$ of $\mathrm{G}$ so that
    $\hat\pi(X\times K)>0$. By induction one gets a sequence $\{h_n\}$ of elements of $\mathrm{H}$ such that
    $Kh_n\cap (Kh_{n-1}\cup\dots\cup Kh_1)=\emptyset$ for all $n\in\mathbb{N}$.
    Thus the sets $\Phi_{h_n}(X\times K)$ are pairwise disjoint for all $n$.
    Since  $\Phi_{h_n}$ preserves  $\hat\pi$ because of $h_n\in \mathrm{H}$,  then
    all these sets have the same positive measure. This implies
    that $\cup_n \Phi_{h_n}(X\times K)$ has infinite
    measure which is a contradiction.

    Now, we will construct the measurable function $P:X\to
    \mathrm{G}$. Consider
    $$
    \mathcal{E}=
    \{(x,g)\in X\times \mathrm{G}: \
    \text{generic point for $\hat\pi$}\} \quad \text{and set} \quad
    \mathcal{E}_x=\{M: (x,g)\in \mathcal{E}\},
    $$
    where a generic point $(x,g)$ in $X\times \mathrm{G}$ of
    the ergodic $(\bar{f},A)$-stationary measure $\hat\pi$
    means that
    $$
    \frac{1}{N}\sum_{n=0}^{N-1} \varphi(F^n_\omega(x,g))
    \to \int \varphi \, d\hat\pi  \quad \text{$\mathbb{P}$-almost surely}
    $$
    for all continuous maps $\varphi:X\times \mathrm{G}\to \mathbb{R}$ with compact support
    where
    $$
    F^n_\omega(x,g)=(f^n_\omega(x),A^n(\bar{x})g), \quad
    \text{$\bar{x}=(\omega,x)\in \Omega\times X$ and $g\in
        \mathrm{G}$, \ \ $n\geq 0$.}
    $$

    \begin{claim}
        If $g,\ell\in \mathcal{E}_x$, then $h=g^{-1}\ell\in \mathrm{H}$.
    \end{claim}
    \begin{proof}
        Since $\Phi_hF_i=F_i\Phi_h$ for all $i\in Y$ it holds that
        $$
        \frac{1}{N}\sum_{n=0}^{N-1} \varphi(F^n_\omega(x,\ell)) =\frac{1}{N}\sum_{n=0}^{N-1} \varphi\circ F^n_\omega(\Phi_h(x,g))
        =\frac{1}{N}\sum_{n=0}^{N-1} \varphi\circ
        \Phi_h(F^n_\omega(x,g)).
        %\to \int \varphi \, d\hat\pi
        $$
        As both $(x,\ell)$ and $(x,g)$ are generic point for $\hat\pi$ then
        taking limit we get that
        $$
        \int \varphi \, d\pi =  \int \varphi\circ \Phi_h \, d\hat{\pi}
        = \int \varphi \, d\Phi_h\hat{\pi}
        $$
        for all continuous maps $\varphi:X\times \mathrm{G}\to \mathbb{R}$ with compact support.
        This implies that $\Phi_h\hat{\pi}=\hat{\pi}$ and thus $h\in \mathrm{H}$.
     \end{proof}

    For $\mu$-almost every $x\in X$, let $P(x)\in \mathcal{E}_x$
    chosen in a measurable way. This follows
    from~\cite[Corollary~18.7]{kechris2012classical} since
    $d\hat{\pi}=\pi_x \, d\mu(x)$ and $\pi_x(\mathcal{E}_x)=1$ for
    $\mu$-almost every $x\in X$. Now we have that
    $$
    \text{$(x,P(x))\in \mathcal{E}$ \ \ and \ \
        $(f_\omega(x),A(\omega,x)P(x))=F_\omega(x,P(x))\in \mathcal{E}$
        \ \ for $\mathbb{P}$-almost every $\omega\in \Omega$.}
    $$
    By definition $(f_\omega(x),P(f_\omega(x)))\in\mathcal{E}$ and thus
    both
    $A(\omega,x)P(x)$ and $P(f_\omega(x))$ belong to $\mathcal{E}_{f_\omega(x)}$. From the
    above claim $P(f_\omega(x))^{-1}A(\omega,x)P(x)\in \mathrm{H}$.
    This complete the proof.
\end{proof}

\subsubsection{Step 3: Proof of Theorem~\ref{thm:measure-product-group}}
    We start by considering the closed group $\mathrm{H}_0=\mathrm{G}$
    and the random cocycle $(\bar{f},A)$ acting on $X \times
    \mathrm{H}_0$ by means of
    $$
    (\bar{f},A): X \times \mathrm{H}_0 \to X \times \mathrm{H}_0,
    \quad (x,h) \mapsto (f(x),A(x)h).
    $$
    By Proposition~\ref{medida}, we find an ergodic
    $(\bar{f},A)$-stationary probability measure $\hat\pi_0$ on
    $X\times \mathrm{H}_0$.
     If $\mathrm{H}(\hat\pi_0)\not=\mathrm{H}_0$ we can
    cohomologically
    reduce the random $\mathrm{H}_0$-valued cocycle $(\bar{f},A)$ by means
    of Proposition~\ref{lem-reduction} to a random
    $\mathrm{H}_1$-valued cocycle $(\bar{f},A_1)$ where
    $\mathrm{H}_1\eqdef \mathrm{H}(\hat{\pi}_0)$.  Arguing inductively
    we can get more. Pick again an ergodic $(\bar{f},A_1)$-stationary
    probability measure $\hat\pi_1$ on $X\times \mathrm{H}_1$.
    Reduce the random $\mathrm{H}_1$-valued
    cocycle $(\bar{f},A_1)$ to a random $\mathrm{H}_2$-valued cocycle
    $(\bar{f},A_2)$ in the case that $\mathrm{H}_2\eqdef
    \mathrm{H}(\hat\pi_1) \lneq \mathrm{H}_1$. This induction defines
    a partial order and thus by Zorn's lemma we can find a minimal
    random $\mathrm{H}$-valued cocycle $(\bar{f},B)$ which cannot be
    reduced. This implies that $\mathrm{H}(\hat\pi)=\mathrm{H}$ for
    every ergodic $(\bar{f},B)$-stationary probability measure
    $\hat\pi$ on $X\times \mathrm{H}$ . Thus, for
    every $h\in \mathrm{H}$, it holds that $\Phi_h\hat\pi=\hat\pi$. In
    particular, if $d\hat\pi=\pi_x \, d\mu(x)$, then $h\pi_x=\pi_x$ for
    $\mu$-almost every $x\in X$, i.e.,
    $$
    \pi_x(E)=h\pi_x(E)=\pi_x(Eh^{-1}) \quad \text{for all Borel
        subsets $E\subset \mathrm{H}$ and $h\in \mathrm{H}$}.
    $$
    Consequently $\pi_x$ is a right-translation-invariant probability
    measure on $\mathrm{H}$. Since  $\mathrm{H}$ is a compact group
    then $\pi_x$ is the Haar measure of $\mathrm{H}$ for $\mu$-almost
    every $x\in X$. Therefore, $\hat{\pi}=\mu \times m_H$ where $m_H$
    is the Haar measure concluding the proof.

    On the other hand, let $\hat\pi_1$ and $\hat\pi_2$ be two ergodic
    $(\bar{f},B)$-stationary measures on $X\times \mathrm{G}$.
    Since both measure project on the same
    measure $\mu$ we can take two generic points of $\hat\pi_1$ and
    $\hat\pi_2$ of the form $(x,g_1)$ and $(x,g_2)$ respectively.
    Recall that a generic point is understand in the sense that
    $$
    \frac{1}{N}\sum_{n=0}^{N-1} \varphi(F^n_\omega(x,g_j))
    \to \int \varphi \, d\hat\pi_j  \quad \text{$\mathbb{P}$-almost surely, \ \ $j=1,2$}
    $$
    for all continuous maps
    $\varphi:X\times \mathrm{G}\to \mathbb{R}$
    with compact support where
    $$
    F^n_\omega(y,g)= F_{\omega_{n-1}}\circ\dots \circ
    F_{w_0}(y,g)=(f^n_\omega(y),B^n(\bar{y})g) \quad
    \text{for $\bar{y}=(\omega,y)\in \Omega\times X$, $g\in \mathrm{G}$,
       $n\in\N$.}
    $$
    Take $h=g_1^{-1}g_2$. Since $\Phi_h \circ F_i=F_i\circ \Phi_h$ for
    all $i\in Y$ we get that
    $$
    \frac{1}{N}\sum_{n=0}^{N-1} \varphi(F^n_\omega(x,g_2)) =\frac{1}{N}\sum_{n=0}^{N-1}
    \varphi\circ F^n_\omega(\Phi_h(x,g_1))
    =\frac{1}{N}\sum_{n=0}^{N-1} \varphi\circ
    \Phi_h(F^n_\omega(x,g_2)).
    %\to \int \varphi \, d\hat\pi
    $$
    Taking limit we have that $\Phi_h\hat\pi_1=\hat\pi_2$. Therefore
    we relate any pair of ergodic stationary measure by the map
    $\Phi_h$ for some $h\in\mathrm{G}$. Observe that the measure
    $\hat\pi=\mu\times m_\mathrm{H}$ obtained above it also is a
    $(\bar{f},B)$-stationary ergodic measure on $X\times \mathrm{G}$.
    Since $\hat\pi$ is a product measure we get that also any other
    $(\bar{f},B)$-stationary ergodic measure on $X\times \mathrm{G}$
     must be a product measure. Consequently, we
    get that any $(\bar{f},B)$-stationary probability measure on
    $X\times \mathrm{G}$ is a product measure
    of the form $\mu \times \nu$ with $\nu=m_H \ast \omega$ where
    $\omega$ is a measure on $\mathrm{G}$. This concludes the
    proof of Theorem~\ref{thm:measure-product-group}.
%\end{proof}

%
%\begin{defi}
%    A random $\mathrm{G}$-valued cocycle $(\bar{f},A)$ is called
%    \emph{essentially bounded} if there are a compact set
%    $K$ of $\mathrm{G}$, a constant $\delta>0$ and a subset
%    $E\subset \bar{X}$ of positive $\bar{\mu}$-measure
% satisfying that
% $$
% d(\{n\geq 0: A^n(\bar{x})\in  K\})\geq \delta \quad
% \text{for all $\bar{x}\in E$.}
% $$
%\end{defi}
%It is easy to check that that the new definition coincides with
%the older when $\mathrm{G}=\mathrm{GL}(d)$.
%
%\begin{thm} \label{thm:measure-product-group}
%Let $(\bar{f},A)$ be an essentially bounded random
%$\mathrm{G}$-valued cocycle. Then there is a  random cocycle
%$(\bar{f},B)$ with values on a compact subgroup $\mathrm{H}$ of
%$\mathrm{G}$ such that
%\begin{enumerate}
%\item $(\bar{f},A)$ is cohomologous with $(\bar{f},B)$, and
%\item  $\mu \times
%m_{\mathrm{H}}$ is the unique $(\bar{f},B)$-stationary probability
%measure on $X\times \mathrm{H}$.
%\end{enumerate}
%Moreover, any $(\bar{f},B)$-stationary probability measure
%$\hat\pi$ on $X\times \mathrm{G}$ is a product measure of the form
%$\hat\pi=\mu \times \nu$ with $\nu=m_H \ast \omega$ where $\omega$
%is a measure on $\mathrm{G}$ and $m_\mathrm{H}$ is the Haar
%measure on $\mathrm{H}$. In particular, every stationary
%probability measure is, in fact, invariant.
%\end{thm}

\subsection{Equivalent definitions of essentially bounded}
%Proposition~\ref{lem-reduction} shows that if a random
%$\mathrm{G}$-valued cocycle $(\bar{f},A)$ admits a stationary
%measure on $X\times \mathrm{G}$ then it is cohomologous with a
%random cocycle $(\bar{f},B)$ with values on a compact group.
%Conversely we have the following:
%
%\begin{thm} \label{thm:equiv}
%    Let $(\bar{f},A)$ be a random $\mathrm{G}$-valued cocycle.
%    Then, it is  equivalent: \vspace{-0.1cm}
%    \begin{enumerate}
%        \item \label{item-eb} $(\bar{f},A)$ is essentially bounded,
%        \item \label{item-st} there is a $(\bar{f},A)$-stationary probability measure on $X\times
%        \mathrm{G}$,
%        \item \label{item-cp} $(\bar{f},A)$ is cohomologous to a random cocycle with values on a compact subgroup  of $\mathrm{G}$.
%    \end{enumerate}
%\end{thm}
Now, we will prove Theorem~\ref{thm:equiv}.
%\begin{proof}
    By Proposition~\ref{medida} we have \eqref{item-eb} implies
    \eqref{item-st}. Also \eqref{item-st} implies \eqref{item-cp}
    follows from Proposition~\ref{lem-reduction}. To complete the
    equivalence we will see \eqref{item-cp} implies \eqref{item-eb}.

    Let $(\bar{f},B)$ be a random cocycle cohomologous to
    $(\bar{f},A)$ with values in a compact subgroup $\mathrm{H}$ of
    $\mathrm{G}$. So there is $P:X\to \mathrm{G}$ measurable such that
    $$
    B^n(\bar{x})=P(f^n_\omega(x))^{-1}A^n(\bar{x})P(x) \in \mathrm{H} \quad
    \text{for $\bar{\mu}$-almost every $\bar{x}=(\omega,x)\in
        \bar{X}=\Omega\times X$ and $n>0$.}
    $$
  By regularity of the measure $P{\mu}$ on $\mathrm{G}$,
  there exists a compact subset $K$ of $\mathrm{G}$ such that the set
  $E=\{\bar{x}=(\omega,x)\in\bar{X}:\, P(x)\in K\}$ has
  $\bar{\mu}$-positive measure. Then, for any integer $n>0$,
  if both, $\bar{x}=(\omega,x)$ and $\bar{f}^n(\omega,x)=(\theta^n\omega,f_\omega^n(x))$
  belong to $E$, then
  $ A^n(\bar{x})=P(f^n_\omega(x))B^n(\bar{x})P(x)^{-1}$
  is in compact set
  $L=\{abc^{-1}: \,  a\in K, b\in H, c\in K\}$.
  Setting $\delta=\bar{\mu}(E)>0$, we have by Birkhoff Theorem
  that
  $
  d(\{n\in\N: \, \bar{f}^n(\bar{x})\in E\})=\delta
  $ %\quad \text{
  for $\hat{\mu}$-almost every $\bar{x}$ in $E$
  and hence $d(\{n\in\N: \, A^n(\bar{x})\in L\})\geq\delta$.
Thus, the random cocycle $(\bar{f},A)$ is essentially bounded.
%\end{proof}

This completes the proof of Theorem~\ref{thm:equiv}.

\begin{rem} \label{rem:equiv}
According to Theorem~\ref{thm:measure-product-group} every
stationary measure of an essentially bounded random cocycle is in
fact invariant. Then, we also have that $(\bar{f},A)$ is
essentially bounded if and only if there is a
$(\bar{f},A)$-invariant probability measure $\eta$ on
$\bar{X}\times \mathrm{G}$ of the form
$\eta=\mathbb{P}\times\hat{\mu}$ where $\hat{\mu}$ is a
probability measure on $X\times\mathrm{G}$.
\end{rem}

Here we relate our results with other literature.

\begin{rem} \label{rem:literatura}
A $\mathrm{G}$-valued cocycle $(f,A)$ is said to be
   \emph{bounded}
    if for every $\varepsilon>0$ there is a compact set $K \subset
    \mathrm{G}$ such that  $ \mu(\{x\in X: A^n(x)\in
    K\})>1-\varepsilon$ for all $n>0$.
    According to~\cite{S81}, a cocycle $(f,A)$ is bounded
    if and only if it is cohomologous to a cocycle with values in a
    compact subgroup of $\mathrm{G}$. Moreover, the equivalence
    between a bounded cocycle $(f,A)$ and the existence
    of a $(f,A)$-invariant probability measure on $X\times \mathrm{G}$
    follows from~\cite{Os96}.
\end{rem}

\begin{rem} Theorem~\ref{thm:equiv} does not follow from
Remark~\ref{rem:literatura}. Indeed, notice that by definition a
$(\bar{f},A)$-invariant probability  $\eta$ on $\bar{X}\times
\mathrm{G}$
     is a measure that $\eta$ projects on $\bar{\mu}=\mathbb{P}\times \mu$
    and has a disintegration $d\eta=\nu_{\bar{x}} \, d\bar{\mu}(\bar{x})$
    such that
    $A(\bar{x})\nu_{\bar{x}}=\nu_{\bar{f}(\bar{x})}$ for
    $\bar{\mu}$-almost every $\bar{x}\in \bar{X}$. However,
    $\eta$ does not necessarily projets on $X\times \mathrm{G}$
    over a $(\bar{f},A)$-stationary measure. Similarly,
    a random cocycle $(\bar{f},A)$ could be cohomologous to a cocycle
    $(\bar{f},B)$  with values in a
    compact subgroup but not necessarily random. Thus, \emph{a priori},
    a bounded random cocycle
    is not necessarily an essentially bounded random cocycle.
\end{rem}

\begin{rem} \label{rem:literatura2}
Theorem~\ref{thm:equiv} includes the results in the literature
given in Remark~\ref{rem:literatura} for deterministic cocycles.
This follows from the fact that in this case (being $\Omega$ a
one-point set) the notion of stationary and invariant measures
coincides (see also Remark~\ref{rem:equiv}). Thus, \emph{a
posteriori}, the notions of essentially bounded and bounded
cocycle are also equivalent.
\end{rem}

\subsection{Invariant measure for essentially bounded linear cocycles}

In this subsection we are going to classify the invariant measures
of essentially bounded random linear cocycles. The following
result proves Theorem~\ref{main-cocycle} for essentially bounded
cocycles. We will prove it as a consequence of
Theorem~\ref{thm:measure-product-group}.

\begin{prop}\label{SLd-reduc}
Let $(\bar{f},A)$ be an essentially bounded random
$\mathrm{GL}(d)$-valued cocycle. Then, there is a  random linear
cocycle $(\bar{f},B)$ with values on a compact subgroup of
$\mathrm{GL}(d)$ and cohomologous to $(\bar{f},A)$ such that every
$ (\bar{f},B)$-stationary probability measure $\hat{\mu}$ on
$X\times \mathbb{P}(\mathbb{R}^d)$  is of the form $\hat{\mu}=
\mu\times \nu$, where $\nu$ is a probability  measure on
$\mathbb{P}(\mathbb{R}^d)$. In particular, every stationary
measure is, in fact, invariant.
\end{prop}

%Proposition~\ref{SLd-reduc} is a consequence of the next general
%result on random cocycles with values in a locally compact
%topological group $\mathrm{G}$. Notice that a random
%$\mathrm{G}$-valued cocycle $(\bar{f},A)$ canonically acts on
%$\bar{X}\times G$ by $(\bar{x},M)\mapsto (f(x),A(x)M)$. We are
%interested in the invariant measures of this action. Observe that
%in the particular case $\mathrm{G}=\mathrm{GL}(d)$, we now look at
%the action of $(\bar{f},A)$ on $\bar{X}\times \mathrm{GL}(d)$
%instead of $\bar{X}\times \mathbb{P}(\R^d)$. This will not be a
%big issue, we will handle it later. First we need to extend the
%notion of essentially bounded cocycle.

%Before proving this theorem, let us explain how it implies
%Proposition~\ref{SLd-reduc}.

\begin{proof}%[Proof of Proposition~\ref{SLd-reduc}]
    By Theorem~\ref{thm:measure-product-group} and
    Remark~\ref{rem:O(2)} we have a random $\mathrm{O}(d)$-valued
    cocycle $(\bar{f},B)$ cohomologous to $(\bar{f},A)$ such that any
    $(\bar{f},B)$-stationary probability on $X\times \mathrm{GL}(d)$
    is a product measure. Let $\hat{\mu}$ be a
    $(\bar{f},B)$-stationary measure on $X\times
    \mathbb{P}(\mathbb{R}^d)$. We write
    $$
    d\hat{\mu}=\nu_x\,d\mu(x) \ \ \text{where
    $\nu_x$ is a measure on $\mathbb{P}(\mathbb{R}^d)$.}
    $$
    Seeing
    $\mathbb{P}(\mathbb{R}^d)$ as the unit $(d-1)$-sphere in
    $\mathbb{R}^d$ with antipodal points identified, we can pull back $\nu_x$ on $\mathbb{S}^d$ to obtain a measure $\nu_x'$ on $\mathbb{S}^{d-1}$ defined as the unique measure on $\mathbb{S}^{d-1}$ invariant by $h\to -h$ and projecting on $\nu_x$ by the canonical projection $\mathbb{S}^{d-1}\to \mathbb{P}(\mathbb{R}^d)$ (more precisely, for any measurable bounded map $\phi:\mathbb{S}^{d-1}\rightarrow \R$ , the even function $h\mapsto \frac{1}{2}(\phi(h)+\phi(-h))$ naturally defines a map $\phi': \mathbb{P}(\mathbb{R}^d)\rightarrow \R$ and then  $\int_{\mathbb{S}^{d-1}} \phi d\nu_x'=\int_{\mathbb{P}(\mathbb{R}^d)}\phi'd\nu_x$). The measure $\hat{\mu}'=\nu_x' d\mu(x)$ is also stationary for the action of $(\overline{f},B)$ on $X\times \mathbb{S}^{d-1}$.\\

    We also denote by $m$ the
    normalized Lebesgue measure on $\mathbb{S}^{d-1}$. We will
    see $\nu_x'$ and $m$ as a probability measures on $\mathbb{R}^d$
    supported on $\mathbb{S}^{d-1}$. For each $x\in X$, we define the
    probability measure
    $$
    \pi_x=\nu_x'\times m \,\times
    \stackrel{d-1}{\dots\dots} \times \, m  \quad \text{on  \ \ $\mathbb{R}^d\times\,
        \stackrel{d}{\dots\dots}\,\times\,\mathbb{R}^d$}.
    $$
    Now, we identify $\mathbb{R}^d\times\,
    \stackrel{d}{\dots\dots}\,\times\,\mathbb{R}^d$ with the set $\mathrm{M}(d)$ of
    $d\times d$ matrices with real coefficients. Again, we can see
    $\pi_x$ as a measure on $\mathrm{M}(d)$.
    \begin{claim}
        The measure $\pi_x$ is supported on $\mathrm{GL}(d)$.
    \end{claim}

    \begin{proof}
   We will prove that $\pi_x$-almost every $M$ in $\mathrm{M}(d)$ is
   invertible. Writing $M$  as
        $M=[h_1,\dots,h_d]$, it is equivalent to prove that $\nu_x'\times m \,\times
    \stackrel{d-1}{\dots\dots} \times \, m$ almost every $h_1,...,h_{d}$ are linearly independant. This can be done by induction on $d$: assuming that $\nu_x'\times m \,\times
    \stackrel{d-2}{\dots\dots} \times \, m$ almost every $h_1,...,h_{d-1}$ are linearly independant, thanks to Fubini theorem it is then enough to prove that for any fixed vectors $h_1,...,h_{d-1}$ linearly independant, for $ m$-almost every $h_d$, the family $h_1,...,h_d$ is linearly independant, or equivalently that $h_d\notin \mathrm{span}(h_1,\dots,h_{d-1})$. But this is obvious since the set $\mathrm{span}(h_1,\dots,h_{d-1})\cap \mathbb{S}^{d-1}$ is a strict submanifold of $\mathbb{S}^{d-1}$, hence has $m$-measure null.
    \end{proof}

    Now, we consider the probability measure $\hat{\pi}$ defined as $
    d\hat{\pi}= \pi_x \, d\mu(x)$. We can write
    $$
    \hat\pi=\hat\mu'\times m^{d-1} \quad \text{on \ \ $(X\times
        \mathbb{R}^d)\times (\mathbb{R}^d)^{d-1}$ \  where  \ $m^{d-1}=m\times\,
        \stackrel{d-1}{\dots\dots}\,\times\,m$.}
    $$
    Moreover, according to the above claim $\hat\pi$ can be seen as a
    measure on $X\times \mathrm{GL}(d)$.

    \begin{claim}
        The probability $\hat\pi$ is a $(\bar{f},B)$-stationary measure on
        $X\times \mathrm{GL}(d)$.
    \end{claim}
    \begin{proof}
        For each $i\in Y$ we have that $F_i\equiv(f_i,B_i)$ acts on
        $X\times \mathrm{GL}(d)$ as $F_i(x,M)=(f_i(x),B_i(x)M)$. We write
        $M=[h_1,\dots,h_d]$ with $h_i\in \mathbb{R}^d$. Then
        $B_i(x)M=[B_i(x)h_1,\dots,B_i(x)h_d]$ and
        $$
        B_i(x)\pi_x =
        B_i(x)\nu_x' \times (B_i(x)m)^{d-1} \quad \text{for all $i\in Y$}.
        $$
         Observe that
        since  $B_i(x)\in
        \mathrm{O}(d)$ for $\mu$-almost $x\in X$ then $B_i(x)m=m$. Hence,
        using that $\mu$ is $f_i$-invariant for $p$-almost every $i\in Y$
        we have that
        \begin{align*}
        \int_{Y} F_i\hat\pi \, dp(i)&=
        \int_Y \int_X B_i(x)\nu_x' \times (B_i(x)m)^{d-1} \, df_i^{-1}\mu(x)
        dp(i)  \\
        &= \int_Y \int_X B_i(x)\nu_x' \times m^{d-1} \, d\mu(x) dp(i) =
        \int_{Y} F_i\hat\mu' \, dp(i) \times m^{d-1} \\
        &=  \hat\mu' \times m^{d-1} = \hat\pi.
        \end{align*}
       We have used that $\int_{Y} F_i\hat\mu' \, dp(i)=\hat\mu'$ since $\hat\mu'$ is a
        $(\bar{f},B)$-stationary measure.
    \end{proof}

    From this claim and according to Theorem~\ref{thm:measure-product-group}, $\hat\pi$ is
    a product measure of the form $\hat{\pi}=\mu\times\nu$ with $\nu$
    a measure on $\mathrm{GL}(d)$. This implies that $\pi_x=\nu$ for
    $\mu$-almost every $x\in X$. Hence, since $\pi_x=\nu_x' \times m
    \times\,
    \stackrel{d-1}{\dots\dots}\,\times\,m$, then $\nu_x'$ does not
    depend on $x$ and so does $\nu_x$, thus $\hat\mu$ is also a product measure.
    This completes the proof of Proposition \ref{SLd-reduc}.
\end{proof}

\section{Essentially unbounded cocycles}  \label{sec:unbounded} %\label{Appendix0}
In this section we are going to provide an alternative proof of
Theorem~\ref{main-cocycle} for essentially unbounded random
$\mathrm{GL}(2)$-valued cocycle. To do this, first we explain the
notion on essentially unbounded cocycle in the particular case of
linear cocycles.

We say that a sequence $(u_n)_n$ of real numbers  \emph{converges
essentially to infinity} if for every $K>0$ the lower asymptotic
density of $D_K=\{n\geq 0: u_n \leq K\}$ is zero. That is, if
$$
 d(D_K)\eqdef \liminf_{n\to \infty} \frac{\#\big([0,n]\cap D_K\big)}{n} =0.
$$
Let $(f,A)$ be a linear cocycle over an ergodic preserving
invertible transformation $f$ of a standard Borel probability
space $(X,\mu)$. Notice that the set of points $x\in X$ such that
the sequence $(u_n)_n$ with $u_n=\|A^n(x)\|$ converges essentially
to infinity is $f$-invariant. Thus, by the ergodicity of the
measure $\mu$, this set is either $\mu$-null or $\mu$-conull. This
implies the following dichotomy:
\begin{enumerate}
\item \label{item1} $\|A^n(x)\|$
converges essentially to infinity for $\mu$-almost every $x\in X$;
\item \label{item2} $\|A^n(x)\|$ does not converge essentially to infinity
for $\mu$-almost every $x\in X$.
\end{enumerate}
This dichotomy allows us to classify the linear cocycles as
follows:
\begin{defi}
A linear cocycle $(f,A)$ is said to be \emph{essentially
unbounded} if~\eqref{item1} holds. Otherwise, i.e., in the
case~\eqref{item2}, or equivalently, if there are $K>0$,
$\delta>0$ and a set $E\subset X$ of positive $\mu$-measure such
that $d(\{n\geq 0: \|A^n(x)\| \leq K\})\geq \delta$ for all $x\in
E$, the cocycle $(f,A)$ is called \emph{essentially bounded}.
\end{defi}

It is easy to check that the new definition coincides with
the older when $\mathrm{G}=\mathrm{GL}(d)$. Notice that a random
linear cocycle $(\bar{f},A)$ is a particular case of a linear
cocycle  over  an ergodic $\bar{\mu}$-preserving invertible
transformation $\bar{f}$. In particular, the notions of
essentially bounded and essentially unbounded linear cocycle
applies for random linear cocycles.

We are going to caracterize first the invariant measures of
essentially unbounded deterministic $\mathrm{GL}(2)$-valued linear
cocycles $(f,A)$. To study these measures, we can normalize the
cocycles by dividing $A$ by $|\det A|^{1/2}$, and assume that it
is a random $\mathrm{SL}^\pm(2)$-valued cocycle.
% where $\mathrm{SL}^\pm(d)$ is
%group of $d\times d$ matrices with determinant $\pm 1$.

\begin{thm} \label{thm11}
    Let $(f,A)$ be a essentially unbounded
    $\mathrm{SL}^{\pm}(2)$-valued cocycle.
    Then, there are measurable families of one-dimensional linear
    subspaces $E_1(x)$ and $E_2(x)$ such that any $(f,A)$-invariant
    measure $\hat{\mu}$ on $X\times
    \mathbb{P}(\mathbb{R}^2)$ %which projets on $\mu$
    is of the form
    $$
    d\hat{\mu}=\big(\lambda \delta_{E_1(x)} + (1-\lambda) \delta_{E_2(x)} \big)
    \,   d\mu(x) \quad \text{for some $0\leq \lambda \leq 1$.}
    $$
\end{thm}

\begin{rem} If the linear cocycle satisfies the integrability conditions
    $\log^+ \|A^{\pm 1}\| \in L^1(\mu)$ and has different extremal
    Lyapunov exponents one gets that $E_1=E^+$ and $E_2=E^-$ where
    $\mathbb{R}^2=E^+(x)\oplus E^-(x)$ is the Oseledets
    decomposition. See~\cite[Lemma~5.25]{Via14}.
\end{rem}

%\subsubsection{Proof of Theorem~\ref{thm11}}
Before proving the above theorem, we will need some
estimates. %for a projective linear action.
First, we identify the projective space $\mathbb{P}(\mathbb{R}^2)$
with $\mathbb{S}^1$. Namely we identify the one-dimensional vector
space $E$ with a unitary vector $h\in \mathbb{R}^2$ so that in
polar form $h=e^{i\theta}$ with $\theta\in \mathbb{S}^1\equiv
\mathbb{R}\,\mathrm{mod} \pi$. Then we can see the projective
action of a $\mathrm{SL}^\pm(2)$-matrix $A$ as a map
$$
\text{$A : \mathbb{S}^1\to \mathbb{S}^1$ given by} \
AE=\spn(e^{iA(\theta)}) \ \ \text{where \
    $E=\spn(e^{i\theta})$.}
$$
We endow $\mathbb{S}^1=
\mathbb{R}/\pi \mathbb{Z}$ with its metric $d(\theta_1,\theta_2)=\inf_{k\in\Z} |\theta_1-\theta_2-k\pi|$. Notice that for $\theta_1,\theta_2$ in $\mathbb{S}^1$, the quantity $\sin|\theta_1-\theta_2|$ is well defined (it does not depend of the representatives of $\theta_1$ and $\theta_2$) and satisfies $\sin|\theta_1-\theta_2|\geq \frac{2}{\pi} d(\theta_1,\theta_2)$.
\begin{lem} \label{lemma1} There is $\theta_0 \in \mathbb{S}^1$ such that
    for every $\varepsilon>0$
    $$
    d(A(\theta_1),A(\theta_2)) \leq
    \frac{2\pi^3}{\|A\|^2\varepsilon^2} \quad \text{for all
        $\theta_1,\theta_2\in \mathbb{S}^1$ with  $d(\theta_i,\theta_0) \geq
        \varepsilon$ for $i=1,2$.}
    $$
\end{lem}
\begin{proof}
    Let $h_1$ and $h_2$ be two linearly independent unitary vectors in
    $\mathbb{R}^2$. We denote by $m$ the maximum between $\|Ah_1\|$
    and $\|Ah_2\|$. Set $h=ah_1 + b h_2$ with $a,b\in \mathbb{R}$ and
    $\|h\|=1$. By  Cramer formula we get that $
    |a|, |b| \leq 2/|\det(h_1,h_2)|. % \frac{2}{|\det(h_1,h_2)|}.
    %\quad b=\frac{det(h,h_2)}{det(h_1,h_2)}
    $ Hence, $ \|Ah\|\leq (|a|+|b|)m \leq 4m / |\det(h_1,h_2)|$ and
    thus
    $$
    \|A\|=\sup \{ \|Ah\|: \|h\|=1\}\leq \frac{4m}{|\det(h_1,h_2)|}.
    $$
    This implies that
    $$
    \|A\|\,|\det(h_1,h_2)| \leq 4 \max\{\|Ah_1\|,\|Ah_2\|\} \quad \text{for
        all $h_1,h_2$ in $\mathbb{R}^2$ with $\|h_1\|=\|h_2\|=1$}.
    $$
    Let $h_0=e^{i\theta_0}$ be such that
    $\|Ah_0\|=\min\{\|Ah\|: \|h\|=1\}$  with $\theta_0\in
    \mathbb{S}^1$.
    Notice that for every
    $h=e^{i\theta}$ with $\theta\in \mathbb{S}^1$
    it holds that
    $|\det(h,h_0)|=\sin(|\theta-\theta_0|)\geq 2
    d(\theta,\theta_0)/\pi$. Hence if $d(\theta,\theta_0)\geq
    \varepsilon$,

    $$
    \frac{2}{\pi} \|A\|\varepsilon \leq 4
    \|Ah\| \quad \text{for every
        $h=e^{i\theta}$ with $\theta\in \mathbb{S}^1$.} $$

    Finally, for any $h_1=e^{i\theta_1}$ and
    $h_2=e^{i\theta_2}$ so that
    $d(\theta_j,\theta_0)\geq \varepsilon$ for $j=1,2$ it
    holds that $\|Ah_1\|\geq \frac{\|A\|\varepsilon}{2\pi}$, $\|Ah_2\|\geq \frac{\|A\|\varepsilon}{2\pi}$ and $|\det(Ah_1,Ah_2)=|\det(h_1,h_2)|\leq 1$, hence
    $$
    \frac{2}{\pi} d(A(\theta_1),A(\theta_2)) \leq \sin(|A(\theta_1)-A(\theta_2)|)
    = \frac{|\det(Ah_1,Ah_2)|}{\|Ah_2\|\|Ah_1\|} \leq \frac{4 \pi^2}{\|A\|^2\varepsilon^2}
    $$
    This completes the proof.
\end{proof}

Now we are ready to provide the proof of Theorem~\ref{thm11}.

\begin{proof}[Proof of Theorem~\ref{thm11}]
    Let $\mathcal{M}(\mu)$ be the set of all
    $(f,A)$-invariant Borel
    probability measures on $X\times \mathbb{P}(\R^2)$.
    We want to prove that $\mathcal{M}(\mu)$ is actually a segment.
    Take $\hat{\nu}\in \mathcal{M}(\mu)$ and write
    $d\hat{\nu}= \nu_x \, d\mu(x)$.  Since $f$ is invertible and
    $\hat{\nu}$ is $(f,A)$-invariant then $\nu_{f(x)}=A(x)\nu_x$ for
    $\mu$-almost every $x\in X$. Consider
    $$
    \|\nu_x\|=\sup\{\,\nu_x(\{E\}): E\in \mathbb{P}(\mathbb{R}^2)\,\}.
    $$
    \begin{claim} \label{claim1}
        $\|\nu_x\|\geq 1/2$ for $\mu$-almost every $x\in X$.
    \end{claim}
    \begin{proof}

        We identify $\mathbb{P}(\mathbb{R}^2)$ with $\mathbb{S}^1\equiv
        \mathbb{R}\,\mathrm{mod} \pi$ and see $\nu_x$ as a measure on
        $\mathbb{S}^1$. Let $\varepsilon>0$ be small enough. Consider
        $C_\varepsilon=\varepsilon^{-2}>0$. Set
        $$
        \varphi_\varepsilon(x) = \sup \{\nu_x(I): \ |I|<\varepsilon\},
        \ \
        B_\varepsilon=\{x\in X:  \varphi_\varepsilon(x)< 1/2 \}
        \ \ \text{and}   \ \
        B_\varepsilon(x)=\{n\geq 0: f^{n}(x) \in B_\varepsilon\}.
        $$
        By Birkhoff ergodic theorem if $\mu(B_\varepsilon)>0$, then
        $B_\varepsilon(x)$ has  asymptotic positive density for
        $\mu$-almost every $x\in X$. In particular, $B_\varepsilon(x)$ has
        lower asymptotic positive density for $\mu$-almost every point in
        $B_\varepsilon$. On the other hand, if $x\in B_\varepsilon$ and
        $n\in B_\varepsilon(x)$, then $\|A^n(x)\|\leq C_\varepsilon$.
        Indeed, assume that on the contrary $\|A^n(x)\|>C_\varepsilon$.
        Hence, applying Lemma~\ref{lemma1} there are arcs $I$, $J$ in
        $\mathbb{S}^1$ with $|I|,|J| <\varepsilon$ such that
        $A^n(x)(\mathbb{S}^1\setminus I) \subset J$. Hence
        $$
        1-\nu_x(I)=\nu_x(\mathbb{S}^1\setminus I) \leq
        \nu_x(A^n(x)^{-1}J)=\nu_{f^{n}(x)}(J)
        $$
        and thus
        $$
        1\leq \nu_x(I) + \nu_{f^{n}(x)}(J) \leq
        \varphi_\varepsilon(x) + \varphi_\varepsilon(f^{-n}(x)) <1/2 + 1/2
        =1
        $$
        which is impossible. Therefore, we get that if
        $\mu(B_\varepsilon)>0$, then for $\mu$-almost every $x\in
        B_\varepsilon$ there are a constant $C_\varepsilon>0$ and a set
        $B_\varepsilon(x)$ of positive lower density such that $\|A^n(x)\|
        \leq C_\varepsilon$ for all $n\in B_\varepsilon(x)$. That is,
        $\|A^n(x)\|$ does not converges essentially to infinity for
        $\mu$-almost every $x\in X$. Consequently, since this is contrary
        to the assumption, $\mu(B_\varepsilon)=0$ for all
        $\varepsilon>0$ small enough. This implies that
        $\varphi_\varepsilon(x)\geq 1/2$ for all small $\varepsilon>0$ and
        $\mu$-almost every $x\in X$. From this,
        \begin{equation*}
        \|\nu_x\| = \lim_{\varepsilon\to 0} \varphi_{\varepsilon}(x)
        \geq \frac{1}{2}   \quad \text{for $\mu$-almost every $x\in X$.}
        \hfill \qedhere
        \end{equation*}
    \end{proof}

    Assume now that $\hat{\nu}$ is ergodic. Since $\hat{\nu}$ is
    $(f,A)$-invariant then $\|\nu_{f(x)}\|=\|A(x)\nu_{x}\|=\|\nu_x\|$
    and thus $\varphi(x)=\|\nu_x\|$ is $f$-invariant. By the
    ergodicity, $\|\nu_x\|=C\geq 1/2$ for $\mu$-almost every $x\in X$.
    This implies that $\nu_x$ has either one or two atoms of maximal
    measure. Let $Y$ be the set of $x\in X$ such that $\nu_x$ has only
    one atom. Since $\hat\nu$ is $(f,A)$-invariant is not difficult to
    see that $Y$ must be $f$-invariant. Thus, again by the
    ergodicity of $\mu$, it follows that $Y$ has $\mu$-measure either
    zero or one. Namely,
    \begin{enumerate}[leftmargin=0.6cm]
        \item \label{case1} if $\mu(Y)=0$, then $\nu_x$ has two atoms $E_x$ and $F_x$ for $\mu$-almost $x\in X$.
        Moreover,
        $$
        \nu_x(\{E_x\})=\nu_x(\{F_x\})=\frac{1}{2} \quad \text{and thus}
        \quad \nu_x  =\frac{\delta_{E_x}+\delta_{F_x}}{2}.
        $$
        \item \label{case2} if $\mu(Y)=1$,  then $\nu_x$ has only
        one atom for $\mu$-almost $x\in X$. Moreover,  $E_{f(x)}=A(x)E_x$
        and thus the measure $d\hat{\rho}=\delta_{{E}_x} \, d\mu(x)$ is
        $(f,A)$-invariant. Since $\nu_x \, d\mu(x) \geq C \,\delta_{E_x}
        \, d\mu(x)$ then $\hat\rho$ is absolutely continuous with respect
        to $\hat{\nu}$. By standard arguments, since $\hat\nu$ is ergodic,
        it follows that $\hat{\nu}=\hat\rho$ and therefore
        $\nu_x=\delta_{E_x}$.
    \end{enumerate}
    Now, we will prove that there are at most two $(f,A)$-invariant
    ergodic measures on $X\times\mathbb{P}(\mathbb{R}^2)$.

    If we are in the above case~\eqref{case1}, then we have a measure
    $d\hat{\nu}= \frac{1}{2}(\delta_{E_x}+\delta_{F_x})\, d\mu(x)$.
    Suppose that we have another $(f,A)$-invariant measure of the form
    $d\hat{\rho}=\delta_{E'_x} \,d\mu(x)$ or
    $d\hat{\rho}=\frac{1}{2}(\delta_{E'_x}+\delta_{F'_x})\, d\mu(x)$.
    If $E'_x=E_x$ for $\mu$-almost every $x\in X$, then
    $\delta_{E_x}\,d\mu(x) \leq 2 \, d\hat\nu$ or $\delta_{F'_x} \,
    d\mu(x) \leq 2\, d\hat\nu$. Notice that either $\delta_{E_x}
    d\mu(x)$ or $\delta_{F'_x} \, d\mu(x)$  is a $(f,A)$-invariant
    measure. Then, similar to the above, the ergodicity of $\hat{\nu}$
    implies that either $2\, d\hat{\nu} =\delta_{E_x}\, d\mu(x)$ or
    $2\,d\hat\nu= \delta_{F'_x} \, d\mu(x)$ which, in both cases, is
    impossible. Therefore, $\hat\nu$ and $\hat\rho$ can not have
    projective atoms in common. Consequently, for any $0<
    \varepsilon<1$ the measure $\lambda_x \, d\mu(x)
    =(1-\varepsilon)\,d\hat\nu + \varepsilon\, d\hat\rho$ is
    $(f,A)$-invariant and $\|\lambda_x\| \leq
    \max\{\frac{1}{2}(1-\varepsilon), \varepsilon\}$. Taking
    $\varepsilon>0$ small enough we get that $\|\lambda_x\|<1/2$ which
    contradicts Claim~\ref{claim1}. Observe that in this case we have
    proved in fact that $\mathcal{M}(\mu)=\{\hat{\nu}\}$.

    If we are in the case~\eqref{case2}, we have
    $d\hat{\nu}=\delta_{E_x} \, d\mu(x)$. Suppose that we have another
    two different $(f,A)$-invariant ergodic measures $\hat{\rho}$ and
    $\hat{\lambda}$. By the above observation, necessarily
    $d\hat\rho=\delta_{E^1_x} \, d\mu(x)$ and
    $d\hat{\lambda}=\delta_{E^2_x} \, d\mu(x)$. Then, all the
    projective atoms must be different and consequently, the
    $(f,A)$-invariant
    measure $d\hat\varrho=\varrho_x\,d\mu(x)$ has
    $$
    \|\varrho_x\| <
    \frac{1}{2} \qquad \text{with} \quad
    \varrho_x=\frac{1}{3}(\delta_{E_x}+\delta_{E^1_x}+\delta_{E^2_x}).
    $$
    However, this is impossible according to Claim~\ref{claim1}. Thus, we can only have at most two $(f,A)$-invariant
    ergodic measures.

    This completes the proof of the theorem.
\end{proof}

Finally we  will show Theorem~\ref{main-cocycle} in the case of
essentially unbounded random $\mathrm{GL}(2)$-valued cocycle as
consequence of Theorem~\ref{thm11}.

\begin{prop} \label{prop:unbounded}
Let $(\bar{f},A)$ be an essentially unbounded  random
$\mathrm{GL}(2)$-valued cocycle. Then there is a  random cocycle
$(\bar{f},B)$ cohomologous to $(\bar{f},A)$ such that every $
(\bar{f},B)$-invariant probability measure $\hat{\mu}$ on $X\times
\mathbb{P}(\mathbb{R}^2)$ is of the form $\hat{\mu}= \mu\times
\nu$ where $\nu$ is a probability measure on
$\mathbb{P}(\mathbb{R}^2)$.
\end{prop}

\begin{proof}
Let $(\bar{f},A)$ be a random $\mathrm{GL}(2)$-valued cocycle.
Since the action of $A$ and $A \cdot |\det A|^{-1/2}$ on
$\mathbb{P}(\R^2)$ coincides we can assume that $(\bar{f},A)$ is a
random $\mathrm{SL}^\pm(2)$-valued cocycle. Now, let $\hat{\mu}$
be a $(\bar{f},A)$-invariant measure on $X\times
\mathbb{P}(\mathbb{R}^2)$. % which projects on $\mu$.
That is, $\hat\mu$ is a $(f_i,A_i)$-invariant measure for
$p$-almost every $i\in Y$ of the form $d\hat\mu=\nu_x \, d\mu(x)$
with $\nu_x$ a
measure on $\mathbb{P}(\mathbb{R}^2)$. % for $\mu$-almost surly.
Consider the product measure $\hat{\mu}_0=\mathbb{P}\times
\hat{\mu}$ on $\Omega\times(X\times \mathbb{P}(\mathbb{R}^2))$.
Observe that $\hat{\mu}_0$ is an invariant measure for the
skew-product $F\equiv(\bar{f},A):\bar{X}\times
\mathbb{P}(\mathbb{R}^2)\to \bar{X}\times
\mathbb{P}(\mathbb{R}^2)$. Moreover, since $d\hat{\mu}_0=\nu_x \,
d\mu(x)d\mathbb{P}(\omega)$,  $\hat{\mu}_0$ projects down on
$\bar{X}=\Omega\times X$ over $\bar{\mu}=\mathbb{P}\times \mu$.
According to Theorem~\ref{thm11}, there are $0\leq \lambda\leq 1$
and measurable families of one-dimensional linear subspace
$E_1(\bar{x})$ and $E_2(\bar{x})$ such that
$$
  \nu_x = \lambda \delta_{E_1(\bar{x})}+
  (1-\lambda) \delta_{E_2(\bar{x})}  \quad \text{for $\bar{\mu}$-almost every $\bar{x}=(\omega,x)\in \bar{X}=\Omega\times X$.}
$$
Since $\nu_x$ does not depend on $\omega$, then $E_1$ and $E_2$
does not depend either. Thus,
$$
\nu_x=\lambda \delta_{E_1(x)}+
  (1-\lambda) \delta_{E_2(x)}  \quad
  \text{for $\mu$-almost every
  $x\in X$}.
$$
Then $E_1(x)$ and $E_2(x)$ are $A_i$-invariant linear subspace for
$p$-almost every $i\in Y$, i.e.,
$$
\text{$A_i(x)\{E_1(x),E_2(x)\}=\{E_1(f_i(x)),E_2(f_i(x))\}$ \ \
for $\mu$-almost every $x\in X$.}
$$

Assume first that $\mu$-almost surely, $\mathbb{R}^2=E_1(x)\oplus
E_2(x)$. Hence, we take in a measurable way $P(x)=[U(x), V(x)]$
where $U(x)$ and $V(x)$  are vectors in $E_1(x)$ and $E_2(x)$
respectively. Then for $p$-almost every $i\in Y$,
$$
  \{R_i(x)U(f_i(x)),S_i(x)V(f_i(x))\}=\{ A_i(x) U(x),
  A_i(x)V(x)\}.
$$
Thus, we get that
\begin{align*}
\text{$B_i(x)=P(f_i(x))^{-1}A_i(x)P(x)$ \ \ is of the form \ \ }
\begin{psmallmatrix}
 R_i(x) & 0 \\ 0 & S_i(x)
\end{psmallmatrix}
 \quad \text{or} \quad
 \begin{psmallmatrix}
 0 & R_i(x) \\ S_i(x) & 0
 \end{psmallmatrix}.
\end{align*}
Otherwise, $\mu$-almost surely $\mathbb{R}^2\not=E_1(x)+ E_2(x)$.
Hence, we take in a measurable way the matrix $P(x)=[U(x), V(x)]$
where $U(x)$ is a vector in $E_1(x)=E_2(x)$ and $V(x)$ is other
non-collinear (ortogonal) vector. Then for $p$-almost every $i\in
Y$, we get that
\begin{align*}
\text{$B_i(x)=P(f_i(x))^{-1}A_i(x)P(x)$ \ \ is of the form \ \ }
\begin{psmallmatrix}
 R_i(x) & T_i(x) \\ 0 & S_i(x)
\end{psmallmatrix}.
\end{align*}

Consequently, any $(\bar{f},B)$-invariant measure $\hat\mu$ on
$X\times \mathbb{P}(\mathbb{R}^2)$ will be of the form
 $$
     d\hat\mu = \lambda \delta_{F_1} + (1-\lambda) \delta_{F_2} \, d\mu(x)
     \ \
     \text{for some $0\leq \lambda\leq 1$}
 %    \quad \text{or} \quad
 %   d\hat\mu =  \delta_{F_1} \, d\mu(x).
 $$
where $F_1$ and $F_2$ are linear subspaces generated by canonical
axis on $\mathbb{R}^2$. This completes the proof of the
proposition.
\end{proof}

\appendix
\addtocontents{toc}{\protect\setcounter{tocdepth}{-1}}
\renewcommand{\thesection}{\Alph{section}}
\renewcommand{\theequation}{\thesection.\arabic{equation}}
\setcounter{equation}{0} \setcounter{thm}{0}

%\section{Essentially unbounded cocycles} \label{Apendix-dim2}
%\subsection{Caracterization of the invariant measures}~

\section{Stationary measure} \label{Apendix1}
Let us consider two standard Borel probability spaces $(X,\mu)$
and $(\Omega,\mathbb{P})$ and endow the product space
$\bar{X}=\Omega\times X$  with the product measure
$\bar{\mu}=\mathbb{P}\times \mu$. Consider a
$\bar{\mu}$-preserving measurable skew-product map
$$
  \bar{f}: \bar{X}\to \bar{X}, \quad
  \bar{f}(\omega,x)=(\theta\omega,f_\omega(x))
$$
where $\theta$ is a ergodic $\mathbb{P}$-invariant invertible
continuous transformation of $\Omega$ and $f_\omega:X\to X$ are
continuous $\mu$-preserving maps for $\mathbb{P}$-almost all
$\omega\in\Omega$. Let $Z$ be a locally compact Hausdorff
topological space and consider the induce Borel $\sigma$-algebra.
Again, consider a skew-product map
\begin{equation} \label{eq:apendix1}
   F:\bar{X}\times Z \to \bar{X}\times Z, \quad F(\bar{x},z)=(\bar{f}(\bar{x}),
   A(\bar{x})z)
\end{equation}
where $A(\bar{x}):Z\to Z$ are continuous maps $\mu$-almost surely.
Sometimes we write $F=\bar{f}\ltimes A(\bar x)$ or $F=(\bar{f},A)$
when no confusion can arise to emphasize the base map and the
fiber maps of~$F$. Observe that since $\bar{f}=\theta\ltimes
f_\omega$ is also a skew-product map, we can also rewrite $F=
\theta\ltimes F_\omega$ where $F_\omega= f_\omega\ltimes
A(\omega,x)$. That is, $F= \theta\ltimes f_\omega\ltimes
A(\omega,x)$.

\begin{defi} A measure $\hat\nu$ on $X\times Z$ is
\emph{$(\bar{f},A)$-stationary} if $\hat\nu$ projets down on $X$
over $\mu$ and
$$
     \hat\nu = \int F_\omega\hat\nu \, d\mathbb{P}(\omega).
$$
\end{defi}

Notice that by definition a stationary measure is not necessarily
a probability measure. Denote by $\mathscr{M}(Z)$ be the Banach
space of signed finite (not necessarily probability) Borel
measures on $Z$ with variation norm. This Banach space can be
identified with the dual of $C_0(Z)$, the space  of bounded
continuous real-valued functions on $Z$ vanishing at infinity with
supremum norm. Then we can endow $\mathscr{M}(Z)$ with the
weak$^*$ topology and consider the Borel $\sigma$-algebra induced
by this topology. Let $L^\infty(X;\mathscr{M}(Z))$ be the Banach
space of (equivalence classes of) essentially bounded measurable
mappings from $X$ to $\mathscr{M}(Z)$. Given $\nu: x \mapsto
\nu_x$ in $L^\infty(X;\mathscr{M}(Z))$ define the measure
$\hat\nu$ on $X\times Z$ by
$$
   \hat\nu(E\times B) = \int_E \nu_x(B) \, d\mu(x)
$$
for $E$, $B$ measurable sets on $X$ and $Z$ respectively and
extending to the product $\sigma$-algebra. By definition $\hat\nu$
has as marginal $\mu$ and disintegration $\nu: x \mapsto \nu_x$.
That is, $d\hat\nu= \nu_x\,d\mu(x)$.

Notice that $L^\infty(X;\mathscr{M}(Z))$ can be identified with
the dual of the Banach space
$$
  L^1(X;C_0(Z))\eqdef \{h:X \to  C_0(Z): \  \|h_x\|_\infty \in L^1(X,\mu)
  \}.
$$
We introduce the transfer operator $\mathscr{P}$ on
$L^1(X;C_0(Z))$ defined by
%:L^1(X;C_0(Z))\toL^1(X;C_0(Z))$ given by
%acting on $\varphi \in L^1(X;C_0(Z))$ by
$$
 \mathscr{P}\varphi \eqdef \int \varphi \circ F_\omega \,
 d\mathbb{P}(\omega) \in L^1(X;C_0(Z)) \ \  \text{for $\varphi \in L^1(X;C_0(Z))$. }
$$
It is clear that $\mathscr{P}$ is a bounded linear operator. The
adjoint transition operator $ \mathscr{P}^*$ acts on
$L^\infty(X;\mathscr{M}(Z))$ by taking  $\nu: x  \mapsto \nu_x$ in
$L^\infty(X;\mathscr{M}(Z))$ and defining $\mathscr{P}^*\nu\in
L^\infty(X;\mathscr{M}(Z))$ as
$$
\mathscr{P}^*\nu: x \mapsto  (\mathscr{P}^*\nu)_x=\int
A(\theta^{-1}\omega,f_\omega^{-1}(x)) \nu_{f_\omega^{-1}(x)}
   \,d\mathbb{P}(\omega).
$$
Consequently $\mathscr{P}^*$ is also a bounded linear operator.
Moreover, if $\mathscr{P}^*\nu=\nu$, then $\hat\nu$ is a
$(\bar{f},A)$-stationary measure on $X\times Z$. Indeed,
\begin{align*}
  \int F_\omega \hat\nu(E\times B) \, d\mathbb{P}(\omega) &=
  \int \int_E A(\theta^{-1}\omega,f_\omega^{-1}(x))
  \nu_{f_\omega^{-1}(x)}(B) \, d\mu(x)d\mathbb{P}(\omega) \\
  &=\int_E (\mathscr{P}^*\nu)_x(B) \, d\mu(x)=\int_E
  \nu_x(B)\,d\mu(x) = \hat\nu(E\times B)
\end{align*}
for all $E$, $B$ measurable sets on $X$ and $Z$ respectively.
Hence $\hat\nu$ is a $(\bar{f},A)$-stationary measure.

According to Banach-Alaoglu's theorem
%(see~\cite[Thm.~3.1]{Conway}),
the unit ball in $L^\infty(X;\mathscr{M}(Z))$ is a compact set.
Moreover, since $X$ is a standard probability space, its
$\sigma$-algebra is countably generated. This implies that
$L^1(X;C_0(Z))$ is separable and thus $L^\infty(X;\mathscr{M}(Z))$
is metrizable \cite{Cr86}. Consequently the unit ball in
$L^\infty(X;\mathscr{M}(Z))$ is also sequentially compact. Now,
the existence of stationary probability measures for random
cocycles follows from standard arguments when $Z$ is compact.

\begin{prop} \label{prop-anex1}
Let $(\bar{f},A)$ be skew-product as~\eqref{eq:apendix1} acting on
$\bar{X}\times Z$. If $Z$ is a compact Hausdorff topological
space, then the set of $(\bar{f},A)$-stationary probability
measures on $X\times Z$ is nonvoid.
\end{prop}

\begin{proof}
Denote by $\mathscr{P}(Z)$ the subset of $\mathscr{M}(Z)$ of
probability measure. Notice that
$$
L^\infty(X;\mathscr{P}(Z))=\{\nu\in L^\infty(X;\mathscr{P}(Z)):
\nu_x \in \mathscr{P}(Z) \ \  \text{$\nu$-almost surely} \}
$$
is a convex subset of the unit ball in
$L^\infty(X;\mathscr{M}(Z))$. Moreover, $\mathscr{P}^*$ leaves
$L^\infty(X;\mathscr{P}(Z))$ invariant. Since, by assumption $Z$
is compact, $L^\infty(X;\mathscr{P}(Z))$ is closed and hence
compact in the weak$^*$ topology.  Brouwer's fixed-point theorem
yields the existence of a  $\mathscr{P}^*$-invariant element $\nu
\in L^\infty(X;\mathscr{P}(Z))$. This yields a
$(\bar{f},A)$-stationary measure $\hat\nu$ on $X\times Z$ defined
by $d\hat\nu=\nu_x\,d\mu(x)$ and completes the proof.
\end{proof}

When $Z$ is not compact we can not guarantee in general that the
set of $(\bar{f},A)$-stationary probability measure is nonvoid.
However the following lemma provides a powerful method to find
stationary finite measure.

\begin{lem} \label{lem:Apendix1}
Let $\nu\in L^\infty(X;\mathscr{P}(Z))$. Then, the set of
accumulation point $\eta \in L^\infty(X;\mathscr{M}(Z))$ of the
sequence $(\nu_n)_n$ given by
$$
    \nu_n=\frac{1}{n}\sum_{j=0}^{n-1} \mathscr{P}^{*j}\nu \in L^\infty(X;\mathscr{P}(Z))  \quad
    \text{for $n\in \mathbb{N}$}
$$
is nonvoid.  Moreover, any accumulation point $\eta$ of
$(\nu_n)_n$ defines a $(\bar{f},A)$-stationary finite measure
$\hat\eta$ on $X\times Z$ whose disintegration is $\eta$. % by means of the disintegration $d\hat\eta=\eta_x \, d\mu(x)$.
Consequently $\hat\eta$ is an accumulation point in the weak$^*$
topology of the sequence of probability measures $(\hat\nu_n)_n$
on $X\times Z$ defined by the disintegrations $(\nu_n)_n$.
\end{lem}
\begin{proof}
Since the unit ball of $L^\infty(X;\mathscr{M}(Z))$ is
sequentially compact then we can extract a convergent subsequence
from $(\nu_n)_n$ and thus the set of accumulation points is not
empty. Moreover, any accumulation point belongs to this ball. Thus
$(\nu_n)_x$ is a finite measure $\mu$-almost surly. On the other
hand, by well known arguments the limit $\eta$ of any convergent
sequences of $(\nu_n)_n$ is also $\mathscr{P}^*$-invariant and
thus $\hat\eta$ is a $(\bar{f},A)$-stationary measure on $X\times
Z$.
\end{proof}

\subsection*{Acknowledgements}
During the preparation of this article the first author was
supported by MTM2017-87697-P from Ministerio de Econom\'ia y
Competividad de Espa\~na and CNPQ-Brasil. The authors were also
funded by RFBM (R\'eseau Franco-Br\'esilien en Math\'ematiques).

\bibliographystyle{alpha}
\bibliography{biblio}

\end{document}